%% file: Main-document.tex
\documentclass[12pt]{amsart}

\usepackage{mathrsfs}
\usepackage{amsrefs}
\usepackage{amsthm}
\usepackage{enumerate}
\usepackage{graphicx}
\usepackage{graphicx} 
\usepackage{scrextend}
\usepackage{amsfonts, amsmath, amssymb,amsthm,comment}  
\usepackage{times,enumerate}
\usepackage{mathdots}%
\usepackage{nccmath}
\usepackage{mathtools}
\usepackage{ulem}
\usepackage{enumitem}
\usepackage{multicol}
		{\end{pmatrix}\end{medsize}}%
\usepackage{dsfont,bbm}
\usepackage{rotating}
\usepackage{lscape}
\usepackage{url}
\usepackage{multicol}
\usepackage{thmtools, thm-restate}
\usepackage{blkarray}
\usepackage{multicol}
\usepackage[a4paper,margin=2.5cm]{geometry}
\usepackage{hyperref}
\usepackage{cancel}
\usepackage{multirow}
\usepackage[baseline]{euflag}

\usepackage[section]{algorithm}
\usepackage{algpseudocode}

\usepackage{pifont}
\usepackage{tikz}
\usetikzlibrary{topaths}
\tikzstyle{every picture}+=[remember picture,inner xsep=0,inner ysep=0.25ex]
\usetikzlibrary{calc}

\DeclareFontFamily{U}{mathx}{\hyphenchar\font45 }
\DeclareFontShape{U}{mathx}{m}{n}{
	<5><6><7><8><9><10><10.95><12><14.4><17.28><20.74><24.88> mathx10
}{}
\DeclareSymbolFont{mathx}{U}{mathx}{m}{n}
\DeclareMathAccent{\widecheck}{0}{mathx}{"71}

\newcommand{\block}[1]{
	\underbrace{\begin{matrix} \mathbf{0}_{p,1} & \cdots & \mathbf{0}_{p,1}\end{matrix}}_{#1}
}
\usepackage{kbordermatrix}
\renewcommand{\kbldelim}{(}
\renewcommand{\kbrdelim}{)}
\renewcommand{\kbrowstyle}{\displaystyle}
\renewcommand{\kbcolstyle}{\displaystyle}
\def\VR{\kern-\arraycolsep\strut\vrule &\kern-\arraycolsep}
\def\vr{\kern-\arraycolsep & \kern-\arraycolsep}
\makeatletter
\newcommand*{\sublabel}[1]{%
	\let\old@currentlabel\@currentlabel%
	\renewcommand{\@currentlabel}{\theenumii}%
	\label{#1}%
	\let\@currentlabel\old@currentlabel%
}
\makeatother
\allowdisplaybreaks

\graphicspath{ {./figs/} }

\hypersetup{
	colorlinks,
	linkcolor={blue},
	citecolor={blue},
	urlcolor={blue}
}

\DeclareMathOperator{\Rank}{rank}

\DeclareMathOperator{\Ker}{Ker}
\DeclareMathOperator{\mult}{mult}

\makeatletter
\def\widebreve{\mathpalette\wide@breve}
\def\wide@breve#1#2{\sbox\z@{$#1#2$}%
	\mathop{\vbox{\m@th\ialign{##\crcr
				\kern0.08em\brevefill#1{0.8\wd\z@}\crcr\noalign{\nointerlineskip}%
				$\hss#1#2\hss$\crcr}}}\limits}
\def\brevefill#1#2{$\m@th\sbox\tw@{$#1($}%
	\hss\resizebox{#2}{\wd\tw@}{\rotatebox[origin=c]{90}{\upshape(}}\hss$}
\makeatletter

\newcommand{\RR}{\mathbb R}

\newcommand{\NN}{\mathbb N}

\newcommand{\SSS}{\mathbb S}

\newcommand{\Sym}{\mathbb{S}}

\newcommand{\cS}{\mathcal{S}}

\newcommand{\cM}{\mathcal M}

\newcommand{\cH}{\mathcal H}

\newcommand{\cZ}{\mathcal Z}
\newcommand{\cC}{\mathcal C}

\newcommand{\benu}{\begin{enumerate}}
	\newcommand{\eenu}{\end{enumerate}}
\newcommand{\bop}{\begin{opomba}}
	\newcommand{\eop}{\end{opomba}}

\newcommand{\Bor}{\mathrm{Bor}}

\newcommand{\tr}{\mathrm{tr}}
\newcommand{\supp}{\mathrm{supp}}

\newtheorem{theorem}{Theorem}[section]
\newtheorem{corollary}[theorem]{Corollary}
\newtheorem{lemma}[theorem]{Lemma}

\theoremstyle{definition}

\newtheorem{example}[theorem]{Example}

\definecolor{green-new}{rgb}{0.0, 0.5, 0.0}

\definecolor{cyan}{rgb}{0.0, 0.8, 1.0}

\theoremstyle{definition}

\newtheorem{remark}[theorem]{Remark}

\numberwithin{equation}{section}

\begin{document}

	\numberwithin{equation}{section}

	\title[Matricial Gaussian quadrature rules: nonsingular case]
	{Matricial Gaussian quadrature rules: nonsingular case}

        \author[A. Zalar]{Alja\v z Zalar${}^{1}$}
	\address{Alja\v z Zalar, 
		Faculty of Computer and Information Science, University of Ljubljana  \& 
		Faculty of Mathematics and Physics, University of Ljubljana  \&
		Institute of Mathematics, Physics and Mechanics, Ljubljana, Slovenia.}
	\email{aljaz.zalar@fri.uni-lj.si}
	\thanks{${}^1$Supported by the ARIS (Slovenian Research and Innovation Agency)
		research core funding No.\ P1-0288 and grants No.\ J1-50002, J1-60011.}
    
	\author[I. Zobovi\v c]{Igor Zobovi\v c${}^{2}$}
	\address{Igor Zobovi\v c, 
		Faculty of Mathematics and Physics, University of Ljubljana  \&
		Institute of Mathematics, Physics and Mechanics, Ljubljana, Slovenia.}
	\email{igor.zobovic@imfm.si}
	\thanks{${}^2$Supported by the ARIS (Slovenian Research and Innovation Agency)
		research core funding No.\ P1-0222 and grant No.\ J1-50002.}

	\begin{abstract}
Let $L$ be a linear operator on univariate polynomials of bounded degree,
    mapping into real symmetric matrices, such that its moment matrix is positive definite. It is known that $L$ admits a finitely atomic positive matrix-valued representing measure $\mu$. Any $\mu$ with the smallest sum of the ranks of the matricial masses is called minimal.
 In this paper, we characterize the existence of a minimal representing measure containing a prescribed atom with prescribed rank of the corresponding mass, thus extending a recent result \cite{BKRSV20} for the scalar-valued case. As a corollary, we obtain a constructive, linear algebraic proof of the strong truncated Hamburger matrix moment problem \cite{Sim06} in the nonsingular case. The results will be important in the study of the truncated univariate rational matrix moment problem.
		\looseness=-1
	\end{abstract}

	
	\date{\today}
	\maketitle

    \input{Introduction}
    \input{Preliminaries}

	\input{Main-result}

\input{Examples}

\input{Main-result-generalized-version}
	

\end{document}

%% file: Introduction.tex
\section{Introduction}
	\label{introduction}
In this paper we study matricial Gaussian quadrature rules for a linear operator $L$ on univariate polynomials of bounded degree, mapping into real symmetric matrices,
such that the corresponding moment matrix is positive definite. More precisely, we fix a real number $t$ and a natural number
 $m\in \NN\cup \{0\}$ and characterize, when there is a minimal representing measure for $L$
    containing $t$ in the support
 with the rank of the corresponding mass equal to $m$.
 Apart from being interesting on its own extending a recent result \cite{BKRSV20} from scalars to matrices, the results will be importantly used in the solution to the truncated univariate matrix rational moment problem, analogous to the scalar case \cite{NZ25}.
	
Let $k\in \NN\cup \{0\}$ and $p\in \NN$. 
We denote by $\RR[x]_{\leq k}$ the vector space of univariate polynomials of degree at most $k$ and by $\Sym_p(\RR)$ the set of real symmetric matrices of size $p\times p$.
For a given linear operator
	\begin{equation} 
            \label{def:linear-mapping}
            L:\RR[x]_{\leq 2n}\to \SSS_p(\RR),
\        \end{equation}
denote by $S_i:=L(x^{i})$, $i=0,1,\ldots,2n$, its \textbf{matricial moments} and by
	\begin{equation}
		\label{def:moment-matrix}
		M(n):= (S_{i+j-2})_{i,j = 1}^{n+1} =
		\kbordermatrix{
			& \mathit{1} & X & X^2 & \cdots & X^n\\
			 \mathit{1} & S_0 & S_1 & S_2 & \cdots & S_n\\[0.2em]
			X & S_1 & S_2 & \iddots & \iddots & S_{n+1}\\[0.2em]
			X^2 & S_2 & \iddots & \iddots & \iddots & \vdots\\[0.2em]
			\vdots & \vdots 	& \iddots & \iddots & \iddots & S_{2n-1}\\[0.2em]
			X^n & S_n & S_{n+1} & \cdots & S_{2n-1} & S_{2n}
		}
	\end{equation}
	the corresponding 
        \textbf{$n$--th truncated moment matrix.} 
Assume that $M(n)$ is positive definite.
It is known (see Theorem \ref{th:Hamburger-matricial} below), that $L$
admits 
    a positive $\Sym_p(\RR)$-valued measure $\mu$ (see \eqref{subsec:matrix-measure}), such that
	\begin{equation}
		\label{moment-measure-cond}
		L(p)=\int_{\RR} p\; d\mu\quad \text{for every }p\in \RR[x]_{\leq 2n}.
	\end{equation}
	Every measure $\mu$ satisfying \eqref{moment-measure-cond}
    is a \textbf{representing measure for $L$}.
   
    A representing measure 
    $\mu=\sum_{j=1}^\ell A_j\delta_{x_j}$ for $L$,
    where each $0\neq A_j\in \Sym_p(\RR)$ is positive semidefinite
    and $\delta_{x_j}$ stands for the Dirac measure supported in $x_j$,
    is \textbf{minimal}, if 
    $\sum_{j=1}^\ell \Rank A_j$ 
    is minimal among all representing measures for $L$.
    In this case, \eqref{moment-measure-cond} is equal to
        \begin{equation}
		\label{moment-measure-cond-Gaussian}
		L(p)=\sum_{j=1}^\ell A_j p(x_j),
	\end{equation}
    and \eqref{moment-measure-cond-Gaussian} is 
    a \textbf{matricial Gaussian quadrature rule for $L$}.
    The points $x_j$ are \textbf{atoms} of the measure $\mu$. If $x_1,x_2,\ldots,x_\ell$ are pairwise distinct, 
    then for each $j$, the matrix $A_j = \mu(\{x_j\})$ is the \textbf{mass} of $\mu$ at $x_j$ and its rank is the \textbf{multiplicity} of $x_j$ in $\mu$, which we denote by 
    $\mult_\mu x_j$. If $x$ is not an atom of $\mu$, then $\mult_\mu x:=0$.\\

    The motivation of the paper is to settle the following problem.\\

    \noindent \textbf{Problem.}
    Let $L$ be as in \eqref{def:linear-mapping} such that $M(n)$ (see \eqref{def:moment-matrix}) is positive definite.
    Given $t\in \RR$ and $m\in \NN\cup \{0\}$,
    characterize when there exists a minimal representing measure $\mu$ for $L$ such that $\mult_\mu t=m$.\\

In \cite[Theorem 1.4]{BKRSV20}, the authors solved the scalar version (i.e., $p=1$ in \eqref{def:linear-mapping}) of the Problem in terms of symmetric determinantal representations involving moment matrices. They also showed how to determine other atoms of $\mu$ based on the determinant of some univariate matrix polynomial.
 Their proof uses convex analysis and algebraic geometry,
 while an alternative proof, using moment theory, and an extension to minimal measures with finitely many prescribed atoms, appears in \cite{NZ+}.
 We mention that in \cite{BKRSV20}, a version of the Problem with an atom at $\infty$,
 called \textbf{evaluation at $\infty$}, is also studied. The corresponding quadrature rules are called \textit{generalized Gaussian quadrature rules}.
We also mention that in the scalar case, the restriction to the case where $M(n)$ is positive definite is natural. Namely, if $M(n)$ is positive semidefinite but not positive definite, then the minimal representing measure is uniquely determined \cite[Theorems 3.9 and 3.10]{CF91}. This fact does not generalize to the matrix case and a version of the Problem with positive semidefinite $M(n)$ is relevant. Moreover, it turns out that this version is technically more involved and will be treated in our forthcoming work \cite{ZZ+}.
 \\
 

The main result of the paper is the solution to the Problem above.

\begin{theorem}
		\label{th:mainTheorem}
         Let $n, p \in \NN$
         and
        $L:\RR[x]_{\leq 2n}\to \SSS_p(\RR)$
        be a linear operator
        such that 
        $M(n)$ is positive definite.
        Fix $t\in \RR$ and $m\in \NN\cup\{0\}$.
        Let $\cH:=(S_{i+j-1}-tS_{i+j-2})_{i,j=1}^n.$
        Then the following statements are equivalent:
        \begin{enumerate}
            \item 
            \label{th:mainTheorem-pt1}
            There exists a minimal
            representing measure $\mu$
            for $L$ such that $\mult_\mu t=m$.
            \item 
			\label{th:mainTheorem-pt2}
			$m\leq \Rank \cH-(n-1)p$.
        \end{enumerate}
	\end{theorem}

In the case $m=0$, Theorem \ref{th:mainTheorem} simplifies to the following result.

    \begin{corollary}
		\label{co:corollary-atom-avoidance}
               Let $n, p \in \NN$
               and
        $L:\RR[x]_{\leq 2n}\to \SSS_p(\RR)$
        be a linear operator
        such that 
        $M(n)$ is positive definite.
        Fix $t\in \RR$. Then 
            there exists a minimal 
            representing measure $\mu$
            for $L$ such that $\mult_\mu t=0$.
    \end{corollary}

A simple consequence of Corollary \ref{co:corollary-atom-avoidance} is a solution
to a strong truncated matrix Hamburger moment problem in the nonsingular case
(i.e., the matrix in \eqref{strong-moment-matrix} below is invertible).

    \begin{corollary}
        \label{co:Simonov-pd}
        Let $n_1,n_2, p \in \NN$ and
        $S_k\in \Sym_p(\RR)$ for $k=-2n_1,-2n_1+1,\ldots,2n_2$.
        Assume that the matrix 
        \begin{equation}
        \label{strong-moment-matrix}
            (S_{i+j-2-2n_1})_{i,j=1}^{n_1+n_2+1}
        \end{equation}
        is positive definite.
        Then 
            there exists a
            measure $\mu$
            such that $S_k=\int_\RR x^k d\mu$ for each $k$.
    \end{corollary}

Corollary \ref{co:Simonov-pd} is a special case of \cite[Theorem 3.3]{Sim06}
under the assumption that the matrix in \eqref{strong-moment-matrix} is positive definite.
The techniques in \cite{Sim06} use involved operator theory, by 
studying self-adjoint extensions of certain, not necessarily everywhere defined, linear operator on the finite dimensional Hilbert space of
vector-valued Laurent polynomials. Our contribution is a constructive, linear algebraic proof, in the sense that representing measures can be
easily computed following the steps in the proof of Theorem \ref{th:mainTheorem} (see Examples  \ref{ex:non-uniqueness} and \ref{ex:2}).
To extend Corollary \ref{co:Simonov-pd} to the singular case (i.e., the matrix in \eqref{strong-moment-matrix} is only positive semidefinite and not necessarily definite), Theorem \ref{th:mainTheorem} needs to be extended to the case $M(n)$ is positive semidefinite \cite{ZZ+}.
In the scalar case an alternative proof of \cite[Theorem 3.3]{Sim06} is \cite[Theorem 3.1]{Zal22b}.

Matricial Gaussian quadrature rules have been studied by several authors (e.g., \cite{DD02,DLR96,DS03}).
These works address the question of computing atoms and masses of a representing measure, which
 is uniquely determined after the odd matricial moment $S_{2n+1}$ is fixed.
The formulas are in terms of the roots of the corresponding orthogonal matrix polynomial.
A novelty of our results is that we do not specify $S_{2n+1}$,
but characterize, when there is a suitable $S_{2n+1}$, that leads to a minimal measure containing a prescribed atom with prescribed multiplicity. In the proof, we essentially construct $S_{2n+1}$ with the required properties such that the extended moment matrix
$M(n+1)$, with $\Rank M(n+1)=\Rank M(n)$, has a suitable block column relation (see
Section \ref{sec:support-of-the-measure}).

Recently, a question related to the Problem was studied in \cite{FKM24}. Namely, the authors describe for $t\in \RR$, the set of all possible masses at $t$ over all representing measures for $L$.
In particular, the maximal mass is determined. The focus of our work is on \textit{minimal} representing measures with \textit{fixed multiplicity} of the mass at $t$.
 In a multivariate setting, the set 
of possible atoms in a representing measure has been charactarized in \cite{MS23+},
while the question of possible masses in a given point was studied in \cite{MS23++}.

	\subsection{Reader's guide} 
In Section \ref{sec:prel} we introduce the notation and some preliminary results.
 In Section \ref{sec:posDef} we present proofs of our main results, i.e., 
 Theorem \ref{th:mainTheorem} and Corollaries \ref{co:corollary-atom-avoidance}, \ref{co:Simonov-pd}.
 In Section \ref{sec:examples} we demonstrate the application of Theorem \ref{th:mainTheorem} on numerical examples (see Examples  \ref{ex:non-uniqueness} and \ref{ex:2}).
 In particular, we show that 
 a minimal representing measure containing a prescribed atom with prescribed multiplicity is not unique and that a given atom can be a part of minimal representing measures with different multiplicities.
 Finally, in Section \ref{sec:posDef-generalized}
 we allow the evaluation at $\infty$ and prove a sufficient condition for the existence of a generalized matricial Gaussian quadrature rule containing $\Rank M(n-1)$ real atoms, among which a prescribed atom has a prescribed multiplicity (see Theorem \ref{th:WithAtomInfinity}).
    

%% file: Preliminaries.tex
\section{Preliminaries}
	\label{sec:prel}
	
    Let $m,m_1,m_2\in \NN$.
    We write $M_{m_1,m_2}(\RR)$ for the set of $m_1\times m_2$ real matrices
    and $M_m(\RR)\equiv M_{m,m}(\RR)$ for short.
    For a matrix $A\in M_{m_1,m_2}(\RR)$
    we call the linear span of its columns
    a $\textbf{column space}$ and denote it by $\cC(A)$.
    We denote by $I_{m}$ the identity $m\times m$ matrix and by $\mathbf{0}_{m_1,m_2}$ the zero
    $m_1\times m_2$ matrix, while $\mathbf{0}_m\equiv \mathbf{0}_{m,m}$ for short.
		We use $M_m(\RR[x])$ to denote $m\times m$ matrices over $\RR[x]$. The elements of $M_m(\RR[x])$ 
		are called \textbf{matrix polynomials}.

     Let $p\in \NN$. 
	For $A\in \Sym_p(\RR)$ the notation $A\succeq 0$ (resp.\ $A\succ 0$) means $A$ is positive semidefinite (psd) (resp.\ positive definite (pd)).
	We use $\Sym^{\succeq 0}_p(\RR)$ for the subset  of all psd matrices in $\Sym_p(\RR)$.
		
		Given a polynomial $p(x)\in \RR[x]$, we write 
		$\cZ(p(x)):=\{x\in \RR\colon p(x)=0\}$ for the set of its zeros.

    \subsection{Matrix measures}
    \label{subsec:matrix-measure}
    Let $\Bor(\RR)$
	be the Borel $\sigma$-algebra of $\RR$. 
	We call 
	$$\mu=(\mu_{ij})_{i,j=1}^p:\Bor(\RR)\to \Sym_p(\RR)$$ 
	a $p\times p$ \textbf{Borel matrix-valued measure} supported on $\RR$
    (or \textbf{positive $\Sym_p(\RR)$-valued measure})
	if
	\begin{enumerate}
		\item 
		$\mu_{ij}:\Bor(\RR)\to \RR$ 
		is a real measure for every $i,j=1,2,\ldots,p$ and
		\smallskip
		\item
		$\mu(\Delta)\succeq 0$ for every $\Delta\in \Bor(\RR)$.
	\end{enumerate} 
    
A positive $\Sym_p(\RR)$-valued measure $\mu$ is \textbf{finitely atomic}, if there exists a finite set $M\in \Bor(\RR)$
 such that
$\mu(\RR\setminus M)=\mathbf 0_{p}$
or
equivalently,
$\mu=\sum_{j=1}^\ell A_j\delta_{x_j}$
for some $\ell\in \NN$, $x_j\in \RR$, $A_j\in \Sym_p^{\succeq 0}(\RR)$.
Let $\mu$ be a positive $\Sym_p(\RR)$-valued measure
and 
$\tau:=\tr(\mu)=\sum_{i=1}^p \mu_{ii}$ denote its trace measure. 
A polynomial $f\in \RR[x]_{\leq k}$ is $\mu$-integrable if 
$f\in L^1(\tau)$. 
We define its integral by
$$
\int_\RR f\;d\mu
=
\Big(\int_\RR f\; d\mu_{ij}\Big)_{i,j=1}^p.
$$

\subsection{Riesz mapping}
    Equivalently, one can define $L$ as in \eqref{def:linear-mapping}
by a sequence of its values on monomials $x^i$, $i=0,1,\ldots,2n$. 
Throughout the paper we will denote these values
by $S_i:=L(x^i)$. 
If
    $\mathcal S:=(S_0,S_1,\ldots,S_{2n})\in (\Sym_p)^{2n+1}$
is given, then we denote the corresponding linear mapping on 
$\RR[x]_{\leq 2n}$ by $L_{\mathcal S}$ and call it a \textbf{Riesz 
mapping of $\mathcal S$}. 

	\subsection{Moment matrix and localizing moment matrices}
	  For $n\in \NN$
        and 
        \begin{equation}
            \label{def:sequence}
                \mathcal S:=(S_0,S_1,\ldots,S_{2n})\in (\Sym_p)^{2n+1},
        \end{equation}
        we denote by $M(n)\equiv M_\cS(n)$ as in \eqref{def:moment-matrix}
        its $n$--th truncated moment matrix.
    For $i,j\in \NN\cup \{0\}$, $i+j\leq 2n$, we also write
\begin{equation}
    \label{def:columns-mm}
    \textbf{v}_i^{(j)}
	=\begin{pmatrix}S_{i+r-1}\end{pmatrix}_{r=1}^{j+1}
	=\begin{pmatrix}
		S_{i} \\ S_{i+1} \\ \vdots \\ S_{i+j}
 	\end{pmatrix}
\end{equation}

		Given $f \in \RR[x]_{\leq 2n}$ and a linear operator $L:\RR[x]_{\leq 2n}\to \Sym_{p}(\RR)$, we define an \textbf{$f$--localizing linear operator} by 
		$$L_f:\RR[x]_{\leq 2n - \deg f}\to \Sym_p(\RR),
		\quad
		L_f(g) := L(fg).
		$$
We call the $\ell$-th truncated moment matrix of  $L_f$ the \textbf{$\ell$--th truncated $f$--localizing moment matrix of $L$} and denote it by $\mathcal{H}_f(\ell)$. 
Defining $$S_i^{(f)} := L_f(x^i) = L(fx^i),$$ we have
\begin{equation*}
    \label{localized-truncated}
		\cH_{f}(\ell):=\left(S^{(f)}_{i+j-2} \right)_{i,j=1}^{\ell+1}
					=	\kbordermatrix{
							& \mathit{1} & X & X^2 & \cdots  & X^{\ell} \\[0.2em]
							\textit{1} & S^{(f)}_0 & S^{(f)}_1 & S^{(f)}_2 & \cdots & S^{(f)}_\ell\\[0.2em]
							X & S^{(f)}_1 & S^{(f)}_2 & \iddots & \iddots & S^{(f)}_{\ell+1}\\[0.2em]
							X^2 & S^{(f)}_2 & \iddots & \iddots & \iddots & \vdots\\[0.2em]
							\vdots &\vdots 	& \iddots & \iddots & \iddots & S^{(f)}_{2\ell-1}\\[0.2em]
							X^\ell & S^{(f)}_\ell & S^{(f)}_{\ell+1} & \cdots & S^{(f)}_{2\ell-1} & S^{(f)}_{2\ell}
						}
	\end{equation*}
		In particular, for $f(x) = x-t$ with $t \in \RR$, we have
		\begin{equation*}
			\mathcal{H}_{x-t}(\ell) 
					=	\kbordermatrix{
							& \mathit{1} & X & X^2 & \cdots  & X^{\ell} \\[0.2em]
							\textit{1} & S_1-tS_0 & S_2-tS_1  & S_3-tS_2  & \cdots & S_{\ell+1}-tS_{\ell}\\[0.2em]
							X & S_2-tS_1  & S_3-tS_2  & \iddots & \iddots & S_{\ell+2}-tS_{\ell+1} \\[0.2em]
							X^2 & S_3-tS_2  & \iddots & \iddots & \iddots & \vdots\\[0.2em]
							\vdots &\vdots 	& \iddots & \iddots & \iddots & S_{2\ell}-tS_{2\ell-1} \\[0.2em]
							X^\ell & S_{\ell+1}-tS_{\ell}  & S_{\ell+2}-tS_{\ell+1} & \cdots & S_{2\ell}-tS_{2\ell-1}& S_{2\ell+1}-tS_{2\ell}
						}.
		\end{equation*}

        For $i,j\in \NN\cup \{0\}$, $i+j\leq 2n-\deg f$, we also write
\begin{equation}
    \label{def:columns-loc-mm}
    {(f\cdot \textbf{v})}_i^{(j)}
	:=\begin{pmatrix}S^{(f)}_{i+r-1}\end{pmatrix}_{r=1}^{j+1}
	=\begin{pmatrix}
		S_{i}^{(f)} \\ S^{(f)}_{i+1} \\ \vdots \\ S_{i+j}^{(f)}
 	\end{pmatrix}.
\end{equation}

 \subsection{Solution to the truncated matrix Hamburger moment problem}

        \begin{theorem}
        [{\cite[Theorem 2.7.6]{BW11}}]
        \label{th:Hamburger-matricial}
        Let $n, p \in \NN$ and
        \begin{equation*}
        \mathcal S\equiv \mathcal S^{(2n)}
        =(S_0,S_1,\ldots,S_{2n})\in (\Sym_p(\RR))^{2n+1}
        \end{equation*}
        be a given sequence.
        Then the following statements are equivalent:
        \begin{enumerate}
            \item 
            There exists a  
            representing measure for $\mathcal S$.
            \item 
            There exists a $(\Rank M(n))$--atomic 
            representing measure for $\mathcal S$.
            \item 
                $M(n)$
                is positive semidefinite
                and 
                $
                \cC(\mathbf{v}_{n+1}^{(n-1)})
                \subseteq
                \cC(M(n)),
                $
                where $\mathbf{v}_{n+1}^{(n-1)}$ is as in \eqref{def:columns-mm}.
        \end{enumerate}
        \end{theorem}

\begin{remark}
The truncated matrix Hamburger moment problem was also considered in 
\cite{Bol96,Dym89,DFKMT09}.
\end{remark}

	\subsection{Support of the representing measure} 
    \label{sec:support-of-the-measure}
		Given a matrix polynomial $P(x)=\sum_{i=0}^n x^i P_i \in M_p(\RR[x])$,
		we define the \textbf{evaluation} $P(X)$
        on the moment matrix $M(n)$ (see \eqref{def:moment-matrix}) to be a matrix, obtained by 
		replacing each monomial of $P$ by the corresponding column of $M(n)$ and multiplying 
		with the matrix coefficients $P_i$ from the right, i.e., 
		$$P(X)
		:=
		\sum_{i=0}^n X^i P_i
		{=
			\mathbf{v}_{0}^{(n)} P_0+
			\mathbf{v}_{1}^{(n)} P_1+
			\cdots+
			\mathbf{v}_{n}^{(n)} P_n
		}
		\in M_{(n+1)p,p}(\RR),$$
        where $\mathbf{v}_{i}^{(j)}$
        are as in \eqref{def:columns-mm} above.
		If $P(X)=\mathbf{0}_{(n+1)p,p}$, then $P$ is a \textbf{block column relation} of $M(n)$.

        The following lemma connects $\supp(\mu)$
        of a representing measure $\mu$ for 
        $\cS$ as in \eqref{def:sequence}
        with a block column relation of $M(n)$.

        \begin{lemma}[{\cite[Lemma 5.53]{KT22}}]
        \label{lemma:atoms}
             Let $n, p \in \NN$ and
        $$
        \mathcal S\equiv \mathcal S^{(2n)}
        =(S_0,S_1,\ldots,S_{2n})\in (\Sym_p)^{2n+1}
        $$
        be a given sequence with a representing measure $\mu$.
        If $P(x)=\sum_{i=0}^n x^i P_i \in M_p(\RR[x])$ is a block column relation of $M(n)$, then
     	\begin{equation*}
			\supp(\mu) \subseteq \cZ(\det P(x)).
		\end{equation*}
        \end{lemma}

		\subsection{Evaluation at $\infty$}
        \label{subsec:evaluation-at-infty}
		We recall the definition of the evaluation at $\infty$ from \cite[Definition 1.1]{BKRSV20}. The \textbf{evaluation at $\infty$} is the linear functional $\mathbf{ev}_{\infty}:\RR[x]_{\leq 2n} \to \RR$, defined by
		\begin{equation}
			\label{eq:evalAtInfty}
			\mathbf{ev}_{\infty}\left (\sum_{i=0}^{2n}a_ix^i\right ) = a_{2n}.
		\end{equation}
        Let $L$ be as in \eqref{def:linear-mapping}.
		We say $\mu$ is a \textbf{finitely atomic $(\RR\cup \{\infty\})$--representing measure for $L$}  if it is of the form 
		$$\mu=\sum_{j=1}^\ell \delta_{x_j}A_j + \mathbf{ev}_{\infty}A,$$ 
		where $\ell\in \NN$, $x_j\in \RR$, and $A_j\in\Sym_p^{\succeq 0}(\RR)$, $A\in \Sym_p^{\succeq 0}(\RR)$. If $\sum_{j=1}^{\ell} \Rank A_j + \Rank A = r$, we say that $\mu$ is an \textbf{$r$--atomic $(\RR \cup \infty)$--representing measure} for $L$.

%% file: Main-result.tex
\section{Proofs of Theorem \ref{th:mainTheorem}
and Corollaries \ref{co:corollary-atom-avoidance}, \ref{co:Simonov-pd}
} 
	\label{sec:posDef}

  In the proof of Theorem \ref{th:mainTheorem} we will need 
		{the following lemma on the determinant of a matrix polynomial.

		\begin{lemma}
			\label{le:determinant-of-matrix-valued-polynomial-v1}
			Let \(p,n \in \mathbb{N}\), \(t \in \mathbb{R}\), and let
			$$H(x) = (x-t)\sum_{i=0}^{n}x^i H_i + P_0$$
			be a nonzero matrix polynomial, where \(H_i,P_0 \in M_p(\mathbb{R})\). 
			We define
			$$
			m := \dim \Ker P_0
			\qquad \text{and} \qquad 
			s := \dim \Big(\Ker P_0 \bigcap \bigcap_{i=0}^{n}\Ker H_i \Big).
			$$
			Then
			\begin{equation}
				\label{determinant-H-v1}
				\det H(x) 
				= 
				\left\{
				\begin{array}{rl}
					(x - t)^{m} g(x),&
					\text{if }s=0,\\[0.3em]
					0,&
					\text{if }s>0,
				\end{array}
				\right.
			\end{equation}
			where $0\neq g(x) \in \mathbb{R}[x]$.
		\end{lemma}
		
		\begin{proof}
			Clearly, if $s>0$, there exists a nonzero vector $v\in \RR^p$ such that $H(x)v=\mathbf 0_{p,1}$, which implies 
			\eqref{determinant-H-v1}. From now on we assume that $s=0$.
			Let
			$\mathcal{B} := \{b_1, b_2, \ldots, b_m,b_{m+1}, \ldots, b_p\}$ be a basis of $\RR^p$ 
			such that the set 
			$\{b_1, b_2, \ldots, b_m\}$ is a basis of $\Ker P_0$. 
			Let us define an invertible matrix
			$$
			B := \begin{pmatrix}
				b_1 & b_2 & \cdots & b_p
			\end{pmatrix} \in M_p(\RR).
			$$
			For $i=0, 1, \ldots, n$ define matrices
			\begin{equation}
				\label{structure-of-widetildeHi-v1}
				\widetilde{H}_i \equiv
				\big(\begin{array}{ccccccc}
					\widetilde{h}_i^{(1)} & \widetilde{h}_i^{(2)} & \cdots &
					\widetilde{h}_i^{(m)} & 
					\cdots &          
					\widetilde{h}_i^{(p)}
				\end{array}\big)
				:= H_iB
			\end{equation}
			and let
			\begin{equation}
				\label{structure-of-widetildeP0-v1}
				\widetilde{P}_0 \equiv
				\big(\begin{array}{ccccccc}
					\smash{\block{m}} &
					\widetilde{p}_0^{(m+1)} & 
					\cdots &          
					\widetilde{p}_0^{(p)}
				\end{array}\big)
				:= P_0B,
			\end{equation}
			\smallskip
			
			\noindent where the last equality follows by 
			$b_1, b_2, \ldots, b_m\in \Ker P_0$.
			Define a matrix polynomial
			$$\widetilde{H}(x) := (x-t)\sum_{i=0}^{n}x^i\widetilde{H}_i + \widetilde{P}_0 = H(x)B \in M_p(\RR[x]).$$
			By \eqref{structure-of-widetildeHi-v1} and \eqref{structure-of-widetildeP0-v1}, it follows that the first $m$ columns of $\widetilde{H}(x)$
			are of the form
			$$
			(x-t)\sum_{i=0}^{n}x^i
			\begin{pmatrix}
				\widetilde{h}_i^{(1)} &   
				\widetilde{h}_i^{(2)} & \cdots & 
				\widetilde{h}_i^{(m)}
			\end{pmatrix},
			$$
			while the last $p-m$ columns of
			$\widetilde H(x)$ are equal to
			$$
			(x-t)\sum_{i=0}^{n}x^i
			\begin{pmatrix}
				\widetilde{h}_i^{(m+1)} &   
				\widetilde{h}_i^{(m+2)} & \cdots & 
				\widetilde{h}_i^{(p)}
			\end{pmatrix} + \begin{pmatrix}
				\widetilde{p}_0^{(m+1)} &   
				\widetilde{p}_0^{(m+2)} & \cdots & 
				\widetilde{p}_0^{(p)}
			\end{pmatrix}
			$$
			Observe that the first $m$ columns of $\widetilde{H}(x)$ have a common factor $(x-t)$.
			Using this observation and upon
			factoring the determinant of 
			$\widetilde{H}(x)$
			column-wise we obtain 
			\begin{align*} 
				\det H(x) = \frac{\det \widetilde{H}(x)}{\det B} = (x-t)^mg(x),
			\end{align*}
			which proves \eqref{determinant-H-v1}. Since $s=0$, $g(x) \neq 0$ also holds.
		\end{proof}

	\begin{proof}[Proof of Theorem \ref{th:mainTheorem}]
            Let 
        $\mathbf{v}_i^{(j)}$ be as in 
            \eqref{def:columns-mm} 
        and 
        ${((x-t)\cdot\mathbf{v})}_{i}^{(j)}$
        as in 
            \eqref{def:columns-loc-mm}
        for $f(x)=x-t$.
        Note that $\cH$ from the statement of the theorem is equal
        to $\mathcal{H}_{x-t}(n-1)$.
            First we establish the following claim.\\

            \noindent \textbf{Claim.}
                $\Rank \begin{pmatrix}
				    \mathbf{v}_0^{(n)} & \begin{array}{c}
					\mathcal{H}_{x-t}(n-1) \\
					{\big({((x-t)\cdot\mathbf{v})}_{n}^{(n-1)}\big)}^T
				\end{array}
			\end{pmatrix}
                =(n+1)p.
                $\\

            \noindent\textit{Proof of Claim.}
            	We have 
		\begin{align*}
			&\Rank \begin{pmatrix}
				    \mathbf{v}_0^{(n)} & \begin{array}{c}
					\mathcal{H}_{x-t}(n-1) \\
					{\big({((x-t)\cdot\mathbf{v})}_{n}^{(n-1)}\big)}^T
				\end{array}
			\end{pmatrix} \\
                =&
                \Rank 
                    \begin{pmatrix}
				    \mathbf{v}_0^{(n)} & 
                        \mathbf{v}_1^{(n)}-t\mathbf{v}_0^{(n)} &
                        \mathbf{v}_2^{(n)}-t\mathbf{v}_1^{(n)} &
                        \cdots &
                        \mathbf{v}_n^{(n)}-t\mathbf{v}_{n-1}^{(n)}
			      \end{pmatrix} \\
                          =&
                \Rank 
                    \begin{pmatrix}
				    \mathbf{v}_0^{(n)} & 
                        \mathbf{v}_1^{(n)}&
                        \mathbf{v}_2^{(n)}&
                        \cdots &
                        \mathbf{v}_n^{(n)}
			      \end{pmatrix} \\        
                =&\Rank M(n) = (n+1)p,
		\end{align*}
            where we used that $M(n)$ is positive definite in the last equality.
            \hfill$\blacksquare$\\
            
		Next we prove the implication 
		$\eqref{th:mainTheorem-pt1} \Rightarrow \eqref{th:mainTheorem-pt2}.$
		Let $\mu=\sum_{j=1}^\ell \delta_{x_j}A_j$ be a {representing measure}  
		for $L$ 
		such that the atoms $x_i\in \RR$ are pairwise distinct, $A_i\in \Sym_p^{\succeq 0}(\RR)$,
		$x_1=t$, 
		$\Rank A_1=m$ and
		$\sum_{j=1}^\ell \Rank A_j = (n+1)p$.  
		We compute $S_{2n+1}$ with respect to the measure $\mu$, i.e., $S_{2n+1} := \int_{\RR} x^{2n+1}\; d\mu = \sum_{j=1}^\ell x_j^{2n+1}A_j$.
	By the Claim,
		\begin{align}
			\label{rank-H-K-equals-np}
			np 
			= 
			\Rank \begin{pmatrix}
				\mathcal{H}_{x-t}(n-1) \\
				{\big({((x-t)\cdot\mathbf{v})}_{n}^{(n-1)}\big)}^T
			\end{pmatrix}
			=
			\Rank \begin{pmatrix}
				\mathcal{H}_{x-t}(n-1) & 
                    {((x-t)\cdot\mathbf{v})}_{n}^{(n-1)}
			\end{pmatrix}.
		\end{align}
		Define $\widetilde{\mu}:=\sum_{j=2}^\ell\delta_{x_j}A_j$.
		Note that 
		$\mu=\delta_t A_1+\widetilde{\mu}$
		and 
		$$
		\sum_{j=2}^\ell \Rank A_j = 
		\sum_{j=1}^\ell \Rank A_j -\Rank A_1=
		(n+1)p - m.
		$$
		Let $\widetilde{L}:\RR[x]\to \Sym_p(\RR)$ be a linear operator, defined by 
		\begin{align*}
			\widetilde{L}(x^i) \equiv \widetilde{S}_i := \int_{\RR} x^i\; d\widetilde{\mu} \quad \text{for } i\in \NN\cup\{0\}.
		\end{align*}
        Let $\widetilde{\mathcal{S}}:=(\widetilde{S}_0,\widetilde S_1,\ldots,\widetilde{S}_{2n+2})$.
		We have
		\begin{align}
			\label{rank-inequalities}
			\begin{split}
				(n+1)p-m=\Rank M(n)-m
				&\leq \Rank{M}_{\widetilde{\cS}}(n)\leq \Rank M_{\widetilde{\cS}}(n+1),\\
				\Rank M_{\widetilde{\cS}}(n+1) &\leq \sum_{j=2}^\ell \Rank A_j = (n+1)p - m,
			\end{split}
		\end{align}
		{where the first inequality follows
			from the fact that the difference
			$M(n)-M_{\widetilde{\cS}}(n)$ is a sum 
			of $m$ matrices of rank 1.
			The inequalities
			\eqref{rank-inequalities} imply 
			that 
			$
			(n+1)p-m
			\leq M_{\widetilde{\cS}}(n+1)
			\leq 
			(n+1)p-m,
			$
			whence all inequalities in \eqref{rank-inequalities} must be equalities.
			In particular,}
		\begin{align}
			\label{rank-equality-of-widetildeMn}
			\Rank M_{\widetilde{\cS}}(n) = \Rank M_{\widetilde{\cS}}(n+1) = (n+1)p - m.
		\end{align}
		For every $i\in \NN\cup \{0\}$ we have
		\begin{align}    
			\label{localizing-moments}
			\begin{split}
				\widetilde{S}_{i+1} - t\widetilde{S}_i &= \int_{\RR} (x^{i+1}-tx^i)\; d\widetilde{\mu} \\
				&= \int_{\RR} (x^{i+1}-tx^i)\; d(\widetilde{\mu}+\delta_tA_1) \\ 
				&=\int_{\RR} (x^{i+1}-tx^i)\; d\mu \\
				&= S_{i+1} - tS_i.
			\end{split}
		\end{align}
		Let $\widetilde{\mathcal{H}}_{x-t}(n-1)$ be the $(x-t)$--localizing moment matrix of $\widetilde{L}$, 
        \begin{equation*}
    {\widetilde{\mathbf{v}}}_i^{(n)}
	=\begin{pmatrix}\widetilde S_{i+r-1}\end{pmatrix}_{r=1}^{n+1}
	=\begin{pmatrix}
		\widetilde S_{i} \\ 
            \widetilde S_{i+1} \\ 
            \vdots \\ 
            \widetilde S_{i+n}
 	\end{pmatrix}
    \quad
    \text{and}
    \quad
        {((x-t)\cdot \widetilde{\mathbf{v}})}_n^{(n-1)}
        :=
        \begin{pmatrix}
		\widetilde{S}_{n+1} - t\widetilde{S}_n\\
            \widetilde{S}_{n+2} - t\widetilde{S}_{n+1} \\ 
            \vdots \\ 
            \widetilde{S}_{2n} - t\widetilde{S}_{2n-1}
 	\end{pmatrix}.
\end{equation*}
		By \eqref{localizing-moments}, it follows that
		\begin{equation}
        \label{eq:localizing-mat}
        \widetilde{\mathcal{H}}_{x-t}(n-1) = \mathcal{H}_{x-t}(n-1)
            \quad\text{and}\quad
            {((x-t)\cdot \widetilde{\mathbf{v}})}_n^{(n-1)}
            =
            {((x-t)\cdot {\mathbf{v}})}_n^{(n-1)}
            .
        \end{equation}
		Hence,
		\begin{align}
			\label{rank-equality-implication-1-to-2}
                \begin{split}
                \Rank \begin{pmatrix}
				    {\widetilde{\mathbf{v}}}_0^{(n)} & \begin{array}{c}
					\mathcal{H}_{x-t}(n-1) \\
					{\big({((x-t)\cdot\mathbf{v})}_{n}^{(n-1)}\big)}^T
				\end{array}
			\end{pmatrix}
                &= 
			\Rank \begin{pmatrix}
				    {\widetilde{\mathbf{v}}}_0^{(n)} & \begin{array}{c}
					{\widetilde{\mathcal{H}}}_{x-t}(n-1) \\
					{\big({((x-t)\cdot\widetilde{\mathbf{v}})}_{n}^{(n-1)}\big)}^T
				\end{array}
			\end{pmatrix} \\
                &= \Rank M_{\widetilde{\cS}}(n) 
                \underbrace{=}_{\eqref{rank-equality-of-widetildeMn}} np +(p-m).
                \end{split}
		\end{align}
		It follows from \eqref{rank-H-K-equals-np} and \eqref{rank-equality-implication-1-to-2} that 
            \begin{align}
            \label{eq:rank-15}
            \begin{split}
	           &\Rank \begin{pmatrix}
				    {\widetilde{\mathbf{v}}}_0^{(n)} & \begin{array}{c}
					\mathcal{H}_{x-t}(n-1) \\
					{\big({((x-t)\cdot\mathbf{v})}_{n}^{(n-1)}\big)}^T
				\end{array}
			\end{pmatrix}\\
            =& \Rank \begin{pmatrix} \widetilde a_1 & \widetilde a_2 & \cdots & \widetilde a_{p-m} & 	\begin{array}{c}
					\mathcal{H}_{x-t}(n-1) \\
					{\big({((x-t)\cdot\mathbf{v})}_{n}^{(n-1)}\big)}^T
				\end{array}
			\end{pmatrix},
            \end{split}
		\end{align}
		where $\widetilde a_1, \widetilde a_2, \ldots, \widetilde a_{p-m}$ are $p-m$ columns of the block ${\widetilde{\mathbf{v}}}_0^{(n)}$.
		We have
		\begin{align}
			\label{rank-equality-implication-1-to-2-v2}
                \begin{split}
	          & \Rank \begin{pmatrix}
				    {\widetilde{\mathbf{v}}}_0^{(n)} & \begin{array}{c}
					\mathcal{H}_{x-t}(n-1) \\
					{\big({((x-t)\cdot\mathbf{v})}_{n}^{(n-1)}\big)}^T
				\end{array}
			\end{pmatrix}\\
                \underbrace{=}_{\eqref{rank-equality-implication-1-to-2}}&
                \Rank \begin{pmatrix}
				    {\widetilde{\mathbf{v}}}_0^{(n)} & \begin{array}{c}
					{\widetilde{\mathcal{H}}}_{x-t}(n-1) \\
					{\big({((x-t)\cdot\widetilde{\mathbf{v}})}_{n}^{(n-1)}\big)}^T
				\end{array}
			\end{pmatrix} \\
                          =&
                \Rank 
                    \begin{pmatrix}
				    \widetilde{\mathbf{v}}_0^{(n)} & 
                        \widetilde{\mathbf{v}}_1^{(n)}-t\widetilde{\mathbf{v}}_0^{(n)} &
                        \widetilde{\mathbf{v}}_2^{(n)}-t\widetilde{\mathbf{v}}_1^{(n)} &
                        \cdots &
                        \widetilde{\mathbf{v}}_n^{(n)}-t\widetilde{\mathbf{v}}_{n-1}^{(n)}
			      \end{pmatrix} \\
                          =&
                \Rank 
                    \begin{pmatrix}
				    \widetilde{\mathbf{v}}_0^{(n)} & 
                        \widetilde{\mathbf{v}}_1^{(n)}&
                        \widetilde{\mathbf{v}}_2^{(n)}&
                        \cdots &
                        \widetilde{\mathbf{v}}_n^{(n)}
			      \end{pmatrix} 
                               \\
                \underbrace{=}_{\eqref{rank-equality-of-widetildeMn}}&
                \Rank 
                    \begin{pmatrix}
				    \widetilde{\mathbf{v}}_0^{(n)} & 
                        \widetilde{\mathbf{v}}_1^{(n)}&
                        \widetilde{\mathbf{v}}_2^{(n)}&
                        \cdots &
                        \widetilde{\mathbf{v}}_n^{(n)} &
                        \widetilde{\mathbf{v}}_{n+1}^{(n)}
			      \end{pmatrix} 
                               \\
                =&\Rank 
                    \begin{pmatrix}
				    \widetilde{\mathbf{v}}_0^{(n)} & 
                        \widetilde{\mathbf{v}}_1^{(n)}-t\widetilde{\mathbf{v}}_0^{(n)} &
                        \widetilde{\mathbf{v}}_2^{(n)}-t\widetilde{\mathbf{v}}_1^{(n)} &
                        \cdots &
                        \widetilde{\mathbf{v}}_{n+1}^{(n)}-t\widetilde{\mathbf{v}}_{n}^{(n)}
			      \end{pmatrix} \\
                =&
                \Rank \begin{pmatrix}
				    {\widetilde{\mathbf{v}}}_0^{(n)} &
					\widetilde{\mathcal{H}}_{x-t}(n) 
			\end{pmatrix}\\
                \underbrace{=}_{\eqref{eq:localizing-mat}}&
                \Rank \begin{pmatrix}
				    {\widetilde{\mathbf{v}}}_0^{(n)} &
					\cH_{x-t}(n) 
			\end{pmatrix}.
                  \end{split}
		\end{align}
		Write
		\begin{align*}
			{((x-t)\cdot\mathbf{v})}_{n}^{(n)}
                =: \begin{pmatrix}
				k_1 & k_2 & \cdots & k_{p}
			\end{pmatrix}.
		\end{align*}
		By \eqref{eq:rank-15} and \eqref{rank-equality-implication-1-to-2-v2}, for every $j=1,2,\ldots,p$, there exist $\alpha_1^{(j)}, \alpha_2^{(j)}, \ldots, \alpha_{p-m}^{(j)} \in \RR$ and $v_j \in M_{np,1}(\RR)$, such that
		\begin{align*}
			k_j = \alpha_1^{(j)}\widetilde a_1 + \alpha_2^{(j)}\widetilde a_2 + \ldots + \alpha_{p-m}^{(j)}\widetilde a_{p-m} + \begin{pmatrix}
				\mathcal{H}_{x-t}(n-1) \\
				{\big({((x-t)\cdot\mathbf{v})}_{n}^{(n-1)}\big)}^T
			\end{pmatrix}v_j.
		\end{align*}
		Hence, we have
		\begin{align*}
                \cH_{x-t}(n)
                = \begin{pmatrix} 
				\widetilde a_1 & \widetilde a_2 & \cdots & \widetilde a_{p-m} & 	\begin{array}{c}
					\cH_{x-t}(n-1) \\
					{\big({((x-t)\cdot\mathbf{v})}_{n}^{(n-1)}\big)}^T
				\end{array}
			\end{pmatrix} \begin{pmatrix}
				\mathbf{0}_{p-m,np} & W \\
				I_{np} & V
			\end{pmatrix},
		\end{align*}
		where 
		\begin{align*}
			W := \begin{pmatrix}
				\alpha_1^{(1)} & \alpha_1^{(2)} & \cdots & \alpha_1^{(p)} \\
				\alpha_2^{(1)} & \alpha_2^{(2)} & \cdots & \alpha_2^{(p)} \\
				\vdots & \vdots & \ddots & \vdots \\
				\alpha_{p-m}^{(1)} & \alpha_{p-m}^{(2)} & \cdots & \alpha_{p-m}^{(p)}
			\end{pmatrix} \quad \text{and} \quad V := \begin{pmatrix}
				v_1 & v_2 & \cdots & v_p
			\end{pmatrix}.
		\end{align*}
		Writing $\widetilde a_j =: \begin{pmatrix}
			b_j \\
			c_j
		\end{pmatrix} \in \begin{pmatrix}
			M_{np,1}(\RR) \\
			M_{p,1}(\RR)
		\end{pmatrix}$ for each $j=1,2,\ldots,p-m$, we have
		\begin{align*}
			\begin{pmatrix}
				\mathcal{H}_{x-t}(n-1) & {{((x-t)\cdot\mathbf{v})}_{n}^{(n-1)}}
			\end{pmatrix} = \begin{pmatrix}
				b_1 & b_2 & \cdots & b_{p-m} & \mathcal{H}_{x-t}(n-1)
			\end{pmatrix} \begin{pmatrix}
				\mathbf{0}_{p-m, np} & W \\
				I_{np} & V
			\end{pmatrix}
		\end{align*}
		and hence,
		\begin{align}
			\label{rank-inequality-computation}
			\begin{split}
				&\Rank \begin{pmatrix}
					\mathcal{H}_{x-t}(n-1) &
                        {{((x-t)\cdot\mathbf{v})}_{n}^{(n-1)}}
				\end{pmatrix} \\ 
				=& \Rank \begin{pmatrix}
					\mathcal{H}_{x-t}(n-1) & \begin{pmatrix}
						b_1 & b_2 & \cdots & b_{p-m}
					\end{pmatrix}W + \mathcal{H}_{x-t}(n-1)V
				\end{pmatrix} \\
				=& \Rank \begin{pmatrix}
					\mathcal{H}_{x-t}(n-1) & \begin{pmatrix}
						b_1 & b_2 & \cdots & b_{p-m}
					\end{pmatrix}W
				\end{pmatrix} \\
				\leq& \Rank \mathcal{H}_{x-t}(n-1) + \Rank \begin{pmatrix}
					b_1 & b_2 & \cdots & b_{p-m}
				\end{pmatrix}W \\
			     \leq& \Rank \mathcal{H}_{x-t}(n-1) + p-m.
			\end{split}
		\end{align}
		Using \eqref{rank-H-K-equals-np} in 
		\eqref{rank-inequality-computation} concludes the proof of the implication 
		$\eqref{th:mainTheorem-pt1} \Rightarrow \eqref{th:mainTheorem-pt2}$.\\
		
		It remains to prove the implication \eqref{th:mainTheorem-pt2} $\Rightarrow$ \eqref{th:mainTheorem-pt1}.
		Let us first describe the main idea of the proof.
		The aim is to construct 
		a matrix $S_{2n+1} \in \Sym_p(\RR)$ such that 
		\begin{equation} 
			\label{condition-to-prove}
			\dim
			\Ker 
                \mathcal{H}_{x-t}(n)
			= m.
		\end{equation}
		It will then follow from 
		$$\Ker\begin{pmatrix} 
            \mathcal{H}_{x-t}(n-1) \\ 
            {\big({((x-t)\cdot\mathbf{v})}_{n}^{(n-1)}\big)}^T\end{pmatrix}=\{0\},$$
		that
		there are $m$ columns in 
            ${((x-t)\cdot\mathbf{v})}_{n}^{(n)}$,
         which are in the span of the other columns of
		$\cH_{x-t}(n)$. 
		By the Claim, the matrix
		\begin{equation}
			\label{def:matrix-loc-plus-1}
			\cM:=
                    \begin{pmatrix}
				\mathbf{v}_{0}^{(n)} &
				\begin{array}{c}
					\mathcal{H}_{x-t}(n-1) \\ 
                        {\big({((x-t)\cdot\mathbf{v})}_{n}^{(n-1)}\big)}^T
				\end{array}
			\end{pmatrix}
		\end{equation}
		is invertible.
		This fact and \eqref{condition-to-prove} will imply that there exists a
		polynomial 
		$$
		P_L(x) 
		= 
		(x - t)
		\Big(
		x^nI_p - 
		\sum_{i=0}^{n-1}x^iP_{i,L}
		\Big) - P_0,
		$$ 
		where   
		$P_{i,L}, P_0 \in M_p(\RR)$ and 
		$\Rank P_0 = p-m$, which is a block column relation of the matrix
		$
		M(n+1)=
        \begin{pmatrix} 
			M(n) & \mathbf{v}_{n+1}^{(n)}\\
            \big(\mathbf{v}_{n+1}^{(n)}\big)^T & S_{2n+2}
		\end{pmatrix},
		$
        where $S_{2n+2}$ is uniquely determined by $\Rank M(n)=\Rank M(n+1)$.
		Hence, \begin{equation} 
			\label{eq:determinant}
			\det P_L(x) = (x - t)^{m}g(x),
		\end{equation}
		for some polynomial $g(x) \in \RR[x]$ of degree $(n+1)p -m$
		{with $g(t)\neq 0$}.
		Thus, Theorem \ref{th:mainTheorem}.\eqref{th:mainTheorem-pt1} will follow from \eqref{eq:determinant}, 
        Theorem \ref{th:Hamburger-matricial} and Lemma \ref{lemma:atoms}. Moreover, the constructed measure $\mu$ will satisfy 
        $\mult_\mu t= m$.\\
		
		Let 
		\begin{align} 
			\label{rank-increasement}
                \begin{split}
			k 
                &:= 
			\Rank 
			\begin{pmatrix}
				\mathcal{H}_{x-t}(n-1) & {((x-t)\cdot\mathbf{v})}_{n}^{(n-1)}
			\end{pmatrix} 
			- 
			\Rank \mathcal{H}_{x-t}(n-1)\\
                &=
			np-\Rank\mathcal{H}_{x-t}(n-1),
                \end{split}
		\end{align}
            where we used that $\cM$ from \eqref{def:matrix-loc-plus-1} is invertible in the last equality.
		By assumption \eqref{th:mainTheorem-pt2}, we have
		$$np-\Rank\cH_{x-t}(n-1)+m=k+m\leq p,$$
		or equivalently 
		$$k\leq p-m.$$
		In order to simplify the technical structure of the proof, we make the following modification.
		We permute the columns of 
            ${((x-t)\cdot\mathbf{v})}_{n}^{(n-1)}$
		using a permutation matrix $P\in M_p(\RR)$
		to obtain a matrix 
		\begin{align}
			\label{def:widehatK}
                \begin{split}
			\widehat{((x-t)\cdot\mathbf{v})}_{n}^{(n-1)}
                &\equiv \big(
			\underbrace{\widehat{{((x-t)\cdot\mathbf{v})}}_{n;1}^{(n-1)}}_{\substack{p-m\\\text{columns}}} \;                \underbrace{{\widehat{((x-t)\cdot\mathbf{v})}}_{n;2}^{(n-1)}}_{\substack{m\\\text{columns}}}
			\big)
                \\
                &:=
            {((x-t)\cdot\mathbf{v})}_{n}^{(n-1)}P,
                \end{split}
		\end{align}
		such that
		\begin{align}
			\label{rank-equalities}
                \begin{split}
			\Rank 
			\begin{pmatrix}
				\mathcal{H}_{x-t}(n-1) & 
                    {\widehat{((x-t)\cdot\mathbf{v})}}_{n;1}^{(n-1)}
			\end{pmatrix} 
			&= 
			\Rank 
			\begin{pmatrix}
				\mathcal{H}_{x-t}(n-1) & {((x-t)\cdot\mathbf{v})}_{n}^{(n-1)}
			\end{pmatrix}\\
			&= 
			\Rank \mathcal{H}_{x-t}(n-1) + k = np.
                \end{split}
		\end{align}
		By \eqref{rank-equalities}, it     
		follows that 
		\begin{equation}
			\label{definition-hat-K-2}
			{\widehat{((x-t)\cdot\mathbf{v})}}_{n;2}^{(n-1)} = \begin{pmatrix}
				\mathcal{H}_{x-t}(n-1) & {\widehat{((x-t)\cdot\mathbf{v})}}_{n;1}^{(n-1)}
			\end{pmatrix} J
		\end{equation} 
		for some $J\in M_{(n+1)p -m,m}(\RR)$ 
		or equivalently
		\begin{align}
			\label{rank-equality-v2}
                \begin{split}
			\begin{pmatrix}
				\mathcal{H}_{x-t}(n-1) & 
				{\widehat{((x-t)\cdot\mathbf{v})}}_{n;1}^{(n-1)} & 
				{\widehat{((x-t)\cdot\mathbf{v})}}_{n;2}^{(n-1)}
			\end{pmatrix}  
			\begin{pmatrix}
				-J \\
				I_{m}
			\end{pmatrix}
			&= \mathbf{0}_{np,m}.
                \end{split}
		\end{align}
		We will now define $\widehat{Z} \in \Sym_p(\RR)$ such that 
		$$
		\begin{pmatrix}
			\mathcal{H}_{x-t}(n-1) & {\widehat{((x-t)\cdot\mathbf{v})}}_{n}^{(n-1)} \\
			\big({\widehat{((x-t)\cdot\mathbf{v})}}_{n}^{(n-1)}\big)^T & \widehat{Z}
		\end{pmatrix}				
		\begin{pmatrix}
			-J \\
			I_{m}
		\end{pmatrix}
		=
		\mathbf{0}_{(n+1)p,m}.
		$$ 
		By \eqref{rank-equality-v2},
		it suffices to establish the equality 
		\begin{equation}
			\label{eq:definition-of-widehatZ}
			\Big(
				\big({\widehat{((x-t)\cdot\mathbf{v})}}_{n}^{(n-1)}\big)^T \;\;\; \widehat{Z}
			\Big) 
			\begin{pmatrix}
				-J \\
				I_{m}
			\end{pmatrix}
			= \mathbf{0}_{p,m}.
		\end{equation}
		Let us decompose $\widehat{Z}$ as
		\begin{equation}
			\label{def:Z}
			\widehat{Z} 
			:= 
			\begin{pmatrix}
				\widehat{Z}_1 & \widehat{Z}_2 \\
				\widehat{Z}_2^T & \widehat{Z}_3
			\end{pmatrix},
		\end{equation}
		where 
		$\widehat{Z}_1$, 
		$\widehat{Z}_2$ and
		$\widehat{Z}_3$
		are of sizes 
		$(p-m)\times (p-m)$,
		$(p-m)\times m$
		and 
		$m \times m$,
		respectively.
		In this notation, \eqref{eq:definition-of-widehatZ} becomes
		\begin{equation}
			\label{eq:definition-of-widehatZ-v2}
			\begin{pmatrix}
				\big({\widehat{((x-t)\cdot\mathbf{v})}}_{n;1}^{(n-1)}\big)^T & \widehat{Z}_1 & \widehat{Z}_2\\
				\big({\widehat{((x-t)\cdot\mathbf{v})}}_{n;2}^{(n-1)}\big)^T & \widehat{Z}_2^T & \widehat{Z}_3
			\end{pmatrix} 
			\begin{pmatrix}
				-J\\
				I_{m}
			\end{pmatrix}
			= \mathbf{0}_{p,m}.
		\end{equation}
		We choose $\widehat{Z}_1 \in \Sym_{p-m}(\RR)$ so that
		\begin{equation}
			\label{choice-of-hat-Z1}
			\Rank
			\begin{pmatrix}
				\cH_{x-t}(n-1) & {\widehat{((x-t)\cdot\mathbf{v})}}_{n;1}^{(n-1)}\\
				\big({\widehat{((x-t)\cdot\mathbf{v})}}_{n;1}^{(n-1)}\big)^T & \widehat{Z}_1
			\end{pmatrix}
			=
			np+(p-m)
		\end{equation}
		and define 
		\begin{equation} 
			\label{equality-v3}
			\widehat{Z}_2 := \Big(
				\big({\widehat{((x-t)\cdot\mathbf{v})}}_{n;1}^{(n-1)}\big)^T \;\;\; \widehat{Z}_1
			\Big) J
			\quad\text{and}\quad 
			\widehat{Z}_3 := \Big(
				\big({\widehat{((x-t)\cdot\mathbf{v})}}_{n;2}^{(n-1)}\big)^T \;\;\; \widehat{Z}_2^T
			\Big) J.
		\end{equation}
		By \eqref{equality-v3},
		it is clear that $\widehat{Z}$ satisfies \eqref{eq:definition-of-widehatZ}. It remains to show that 
		$\widehat{Z}$ is symmetric. Since $\widehat{Z}_1\in \Sym_{p-m}(\RR)$, we only need to show that $\widehat{Z}_3\in \Sym_{m}(\RR)$. But this follows by the following computation:
		\begin{align*} 
			\widehat{Z}_3 
			&= \begin{pmatrix}
				\big({\widehat{((x-t)\cdot\mathbf{v})}}_{n;2}^{(n-1)}\big)^T & \widehat{Z}_2^T
			\end{pmatrix}  J 
			= \begin{pmatrix}
				{\widehat{((x-t)\cdot\mathbf{v})}}_{n;2}^{(n-1)} \\
				\widehat{Z}_2
			\end{pmatrix}^T  J\\ 
			&\underbrace{=}_{
				\substack{
					\eqref{definition-hat-K-2},\\
					\eqref{equality-v3}
				}
			} \left(\begin{pmatrix}
				\mathcal{H}_{x-t}(n-1) & {\widehat{((x-t)\cdot\mathbf{v})}}_{n;1}^{(n-1)}\\
				\big({\widehat{((x-t)\cdot\mathbf{v})}}_{n;1}^{(n-1)}\big)^T & \widehat{Z}_1
			\end{pmatrix}
			J \right)^T J\\ 
			&= J^T \begin{pmatrix}
				\mathcal{H}_{x-t}(n-1) & {\widehat{((x-t)\cdot\mathbf{v})}}_{n;1}^{(n-1)} \\[0.2em]
				\big({\widehat{((x-t)\cdot\mathbf{v})}}_{n;1}^{(n-1)}\big)^T & \widehat{Z}_1
			\end{pmatrix}^T J.
		\end{align*}
		Defining 
		the vectors $c_1, c_2, \ldots, c_{m}$
		by 
		$$C \equiv \begin{pmatrix}
			c_1 & c_2 & \cdots & c_{m}
		\end{pmatrix} := (I_{np}\oplus P)
		\begin{pmatrix}
			-J \\
			I_{m}
		\end{pmatrix}$$
		and the matrix $Z$ by 
		\begin{equation} 
			\label{def:Z-pd-case}
			Z:=P\widehat ZP^{T}\in \Sym_p(\RR),
		\end{equation}
		we have
		$$
		c_1, c_2, \ldots, c_{m}
		\in 
		\Ker
		\begin{pmatrix}
			\mathcal{H}_{x-t}(n-1) & {((x-t)\cdot\mathbf{v})}_{n}^{(n-1)} \\
			{\big({((x-t)\cdot\mathbf{v})}_{n}^{(n-1)}\big)}^T & Z
		\end{pmatrix}
		$$
		and $c_1,c_2,\ldots,c_m$ are linearly independent.
		Indeed,
		\begin{align*}
			&\begin{pmatrix}
			\mathcal{H}_{x-t}(n-1) & {((x-t)\cdot\mathbf{v})}_{n}^{(n-1)} \\
			{\big({((x-t)\cdot\mathbf{v})}_{n}^{(n-1)}\big)}^T & Z
		\end{pmatrix}
            C \\ 
            &=\begin{pmatrix}
				\mathcal{H}_{x-t}(n-1) & {\widehat{((x-t)\cdot\mathbf{v})}}_{n}^{(n-1)}P^T \\
				P\big({\widehat{((x-t)\cdot\mathbf{v})}}_{n}^{(n-1)}\big)^T & P\widehat{Z}P^T
			\end{pmatrix}(I_{np}\oplus P)
			\begin{pmatrix}
				-J \\
				I_{m}
			\end{pmatrix}
			\\ &= (I_{np}\oplus P)\begin{pmatrix}
				\mathcal{H}_{x-t}(n-1) & {\widehat{((x-t)\cdot\mathbf{v})}}_{n}^{(n-1)} \\
				\big({\widehat{((x-t)\cdot\mathbf{v})}}_{n}^{(n-1)}\big)^T & \widehat{Z}
			\end{pmatrix}(I_{np}\oplus P^T)(I_{np}\oplus P)
			\begin{pmatrix}
				-J \\
				I_{m}
			\end{pmatrix}
			\\ &= (I_{np}\oplus P)\begin{pmatrix}
				\mathcal{H}_{x-t}(n-1) & {\widehat{((x-t)\cdot\mathbf{v})}}_{n}^{(n-1)} \\
				\big({\widehat{((x-t)\cdot\mathbf{v})}}_{n}^{(n-1)}\big)^T & \widehat{Z}
			\end{pmatrix}
			\begin{pmatrix}
				-J \\
				I_{m}
			\end{pmatrix}
			\\ &\underbrace{=}_{
				\substack{
					\eqref{rank-equality-v2},\\
					\eqref{eq:definition-of-widehatZ}
				}
			} (I_{np}\oplus P)\mathbf{0}_{(n+1)p,m} = \mathbf{0}_{(n+1)p,m}.
		\end{align*}
		Defining 
		$$S_{2n+1} := Z + tS_{2n}\in \Sym_p(\RR),$$
		the equality \eqref{condition-to-prove}
		holds.
		Since $\cM$ from \eqref{def:matrix-loc-plus-1}
		is invertible,
		it follows that 
		\begin{align*}
			\widehat{\mathcal{M}} 
			&:= (I_{np}\oplus P^T)\mathcal{M} = \begin{pmatrix} 
				\widehat{\mathbf{v}}_{0}^{(n)} &
				\begin{array}{c}
					\mathcal{H}_{x-t}(n-1)\\
					\big({\widehat{((x-t)\cdot\mathbf{v})}}_{n}^{(n-1)}\big)^T
				\end{array}
			\end{pmatrix}
		\end{align*}
		is also invertible, where
		\begin{align}
			\label{permutation}
			\begin{split}
				\widehat{\mathbf{v}}_{0}^{(n)}
				&:=
				(I_{np}\oplus P^T)
				\mathbf{v}_{0}^{(n)}.
			\end{split}
		\end{align}
		Therefore
		$$
		\begin{pmatrix}
                {\widehat{((x-t)\cdot\mathbf{v})}}_{n;1}^{(n-1)}\\
			\widehat{Z}_1\\
			\widehat{Z}_2^T
		\end{pmatrix}
		=\widehat{\cM}
		\begin{pmatrix} 
			U_1 \\ U_2
		\end{pmatrix}
		$$
		for some real matrices $U_1\in M_{p,p-m}(\RR)$
		and $U_2\in M_{np,p-m}(\RR)$
		with 
		\begin{equation} 
			\label{rank-of-U1}
			\Rank U_1=p-m \quad (\text{see }\eqref{choice-of-hat-Z1}),
		\end{equation}
		or equivalently
		\begin{equation}
			\label{equality-v2}
			\begin{pmatrix}
				\widehat \cM &
				\begin{array}{c}
					{\widehat{((x-t)\cdot\mathbf{v})}}_{n;1}^{(n-1)}\\
					\widehat{Z}_1\\
					\widehat{Z}_2^T
				\end{array}
			\end{pmatrix}
			\begin{pmatrix}
				-U_1  \\
				-U_2   \\
				I_{p-m}
			\end{pmatrix}
			=\mathbf{0}_{(n+1)p, (p-m)}.
		\end{equation}
		By 
		\eqref{rank-equality-v2}, 
		\eqref{equality-v3}
		and 
		\eqref{equality-v2},
		we have that
		$$
		\begin{pmatrix} 
			\widehat{\mathbf{v}}_{0}^{(n)} &
			\begin{array}{cc}
				\mathcal{H}_{x-t}(n-1) & {\widehat{((x-t)\cdot\mathbf{v})}}_{n}^{(n-1)}\\
				\big({\widehat{((x-t)\cdot\mathbf{v})}}_{n}^{(n-1)}\big)^T & \widehat Z
			\end{array}
		\end{pmatrix} 
		\begin{pmatrix}
			-U_1 & \mathbf{0}_{p,m}\\
			-U_2 & -J_1 \\
			I_{p-m} & -J_2\\
			\mathbf{0}_{m,p-m} & I_m
		\end{pmatrix}
		=\mathbf{0}_{(n+1)p,p},
		$$
		where 
		$J=:
		\begin{pmatrix} J_1 \\ J_2 \end{pmatrix}
		\in 
		\begin{pmatrix}
			M_{np,m}(\RR)\\
			M_{p-m,m}(\RR)
		\end{pmatrix}.
		$
		Therefore
		\begin{align}
			\label{equality-v4}
                \begin{split}
			&\begin{pmatrix}
				{\widehat{((x-t)\cdot\mathbf{v})}}_{n}^{(n-1)} \\
				\widehat{Z}
			\end{pmatrix}
			\begin{pmatrix}
				I_{p-m} & -J_2\\
				\mathbf{0}_{m,p-m} & I_{m}
			\end{pmatrix}\\
			=& 
			\begin{pmatrix}
			    \mathcal{H}_{x-t}(n-1)\\
                    \big({\widehat{((x-t)\cdot\mathbf{v})}}_{n}^{(n-1)}\big)^T
			\end{pmatrix}
                \begin{pmatrix}
				U_2 & J_1
			\end{pmatrix}
                + 
			\widehat{\mathbf{v}}_{0}^{(n)}
			\begin{pmatrix}
				U_1 & \mathbf{0}_{p,m}
			\end{pmatrix}
                \end{split}
		\end{align}
		Using \eqref{permutation} and 
		$$
		\begin{pmatrix}
			{\widehat{((x-t)\cdot\mathbf{v})}}_{n}^{(n-1)} \\
			\widehat{Z}
		\end{pmatrix}  
		=\begin{pmatrix}
			{((x-t)\cdot\mathbf{v})}_{n}^{(n-1)}P \\
			P^TZP
		\end{pmatrix}  
		=
		(I_{np}\oplus P^T){((x-t)\cdot\mathbf{v})}_{n}^{(n)}P
		$$
		in \eqref{equality-v4},
		we get
		\begin{equation}
			\label{equality-v5}
			{((x-t)\cdot\mathbf{v})}_{n}^{(n)}
			G_{n}
			= 
			\sum_{i=0}^{n-1}
			{((x-t)\cdot\mathbf{v})}_{i}^{(n)}G_i + 
			{\mathbf{v}}_{0}^{(n)}
			\begin{pmatrix}
				U_1 & \mathbf{0}_{p,m}
			\end{pmatrix},
		\end{equation}
		where
		$G_{n}
		:=P
		\begin{pmatrix}
			I_{p-m} & -J_2\\
			\mathbf{0}_{m,p-m} & I_{m}
		\end{pmatrix}$
        and 
		$G_0, G_1,\ldots,{G}_{n-1}
		\in M_p(\RR).$
        Since $G_{n}$ is invertible, \eqref{equality-v5}
        is equivalent to 
		\begin{equation*}
			{((x-t)\cdot\mathbf{v})}_{n}^{(n)}
			= 
			\sum_{i=0}^{n-1}
			{((x-t)\cdot\mathbf{v})}_{i}^{(n)}G_iG_{n}^{-1} + 
			{\mathbf{v}}_{0}^{(n)}
			\begin{pmatrix}
				U_1 & \mathbf{0}_{p,m}
			\end{pmatrix}G_{n}^{-1}.
		\end{equation*}
		We now define the matrix polynomial
		\begin{align}
			\label{gen-poly}
			\begin{split}
				H(x) 
				&:= 
				(x^{n+1} - tx^n)I_p - \sum_{i=0}^{n-1}(x^{i+1} - tx^i)G_iG_n^{-1} - 
				\begin{pmatrix}
					U_1 & \mathbf{0}_{p,m}
				\end{pmatrix}G_n^{-1}
				\\
				&= (x - t)\left (x^nI_p - \sum_{i=0}^{n-1}x^iG_iG_n^{-1}\right ) -
				\begin{pmatrix}
					U_1 & \mathbf{0}_{p,m}
				\end{pmatrix}G_n^{-1}.
			\end{split}
		\end{align} 
		Observe that $H(x)$ is monic 
		of degree $n + 1$ and represents the block column relation  $H(X) = \mathbf{0}_{(n+1)p,p}$ in the matrix  
        \begin{equation}
        \label{def:extended-matrix}
		M(n+1)=
        \begin{pmatrix} 
			M(n) & \mathbf{v}_{n+1}^{(n)}\\
            \big(\mathbf{v}_{n+1}^{(n)}\big)^T & S_{2n+2}
		\end{pmatrix},
		\end{equation}
        where $S_{2n+2}$ is uniquely determined by $\Rank M(n)=\Rank M(n+1)$. Note that 
		$$\Rank \begin{pmatrix}
			U_1 & \mathbf{0}_{p,m}
		\end{pmatrix}G_n^{-1} = \Rank \begin{pmatrix}
			U_1 & \mathbf{0}_{p,m}
		\end{pmatrix} \underbrace{=}_{\eqref{rank-of-U1}} p-m,$$ whence 
		$$\dim \big(\Ker \begin{pmatrix}
			U_1 & \mathbf{0}_{p,m}
		\end{pmatrix}G_n^{-1}\big) =m.$$
		{By Lemma} \ref{le:determinant-of-matrix-valued-polynomial-v1} used for $H(x)$ from \eqref{gen-poly}, we get that $$\det H(x) = (x - t)^{m}g(x),$$ for some polynomial $g(x) \in \RR[x]$ of degree $(n+1)p-m$. By Theorem \ref{th:Hamburger-matricial}, there exists a
		representing measure for $L$ of the form 
		$\mu=\sum_{j=1}^\ell \delta_{x_j}A_j$, 
		where $\ell\in \NN$, $x_j\in \RR$ are pairwise distinct, $A_j\in \Sym_p^{\succeq 0}(\RR)$ and $(n+1)p=\Rank M(n)=\sum_{j=1}^\ell \Rank A_j$.
		By Lemma \ref{lemma:atoms}, 
		the atoms $x_1, x_2, \ldots, x_\ell$ are exactly pairwise distinct zeros of $\det H(x)$. Hence, $t = x_{j'}$ for some $j' \in \{1, 2, \ldots, \ell\}$, and $\Rank A_{j'} \geq m$. We now need to show that $\Rank A_{j'} = m$. Suppose on the contrary that $\Rank A_{j'} = m'$ for some 
		\begin{equation} 
			\label{contradiction-argument}
			m' > m. 
		\end{equation}
		Let us define the measure 
		$\widetilde{\mu} := \mu - \delta_tA_{j'}.$
		By analogous reasoning as for $\mu$
		in the proof of implication $\eqref{th:mainTheorem-pt1} \Rightarrow \eqref{th:mainTheorem-pt2}$,
		we obtain an equality of type \eqref{rank-equality-implication-1-to-2-v2}, where $m$ is replaced by $m'$, and $\widetilde a_1, \widetilde a_2, \ldots, \widetilde a_{p-m'}$ are $p-m'$ columns of the block $\widetilde{\mathbf{v}}_{0}^{(n)} := \begin{pmatrix}
			\widetilde{S}_0 & \widetilde{S}_1 & \cdots & \widetilde{S}_n
		\end{pmatrix}$, where $\widetilde{S}_i := \int_{\RR} x^i\; d\widetilde{\mu}$. The equality is equivalent to
		\begin{align*}
            \mathcal{H}_{x-t}(n) 
            = \begin{pmatrix} \widetilde a_1 & \widetilde a_2 & \cdots & \widetilde a_{p-m'} & 	\begin{array}{c}
					\mathcal{H}_{x-t}(n-1) \\
					{\big({((x-t)\cdot\mathbf{v})}_{n}^{(n-1)}\big)}^T
				\end{array}
			\end{pmatrix} \begin{pmatrix}
				\mathbf{0}_{p-m', np} & \widetilde W \\
				I_{np} & \widetilde V
			\end{pmatrix},
		\end{align*}
		for some matrices $\widetilde W \in M_{p-m',p}(\RR)$ and $\widetilde V \in M_{np,p}(\RR)$.
		Hence, we have
		\begin{align}
			\label{rank-inequality-computation-end-of-proof-2to1}
			\begin{split}
				&\Rank \mathcal{H}_{x-t}(n)\\
				=& \Rank \begin{pmatrix}
					\begin{array}{c}
						\mathcal{H}_{x-t}(n-1)  \\
						{\big({((x-t)\cdot\mathbf{v})}_{n}^{(n-1)}\big)}^T
					\end{array} & \begin{pmatrix}
						\widetilde a_1 & \widetilde a_2 & \cdots & \widetilde a_{p-m'}
					\end{pmatrix} \widetilde W + \begin{pmatrix}
						\mathcal{H}_{x-t}(n-1)  \\
						{\big({((x-t)\cdot\mathbf{v})}_{n}^{(n-1)}\big)}^T
					\end{pmatrix}\widetilde V
				\end{pmatrix} \\
				=& \Rank \begin{pmatrix}
					\begin{array}{c}
						\mathcal{H}_{x-t}(n-1)  \\
						{\big({((x-t)\cdot\mathbf{v})}_{n}^{(n-1)}\big)}^T
					\end{array} & \begin{pmatrix}
						\widetilde a_1 & \widetilde a_2 & \cdots & \widetilde a_{p-m'}
					\end{pmatrix}\widetilde W
				\end{pmatrix} \\
				\leq& \Rank \begin{pmatrix}
					\mathcal{H}_{x-t}(n-1)  \\
					{\big({((x-t)\cdot\mathbf{v})}_{n}^{(n-1)}\big)}^T
				\end{pmatrix} + \Rank \begin{pmatrix}
					\widetilde a_1 & \widetilde a_2 & \cdots & \widetilde a_{p-m'}
				\end{pmatrix}\widetilde W \\
				\leq& \Rank \begin{pmatrix}
					\mathcal{H}_{x-t}(n-1)  \\
					{\big({((x-t)\cdot\mathbf{v})}_{n}^{(n-1)}\big)}^T
				\end{pmatrix} + p-m' \\
				=& np + p-m'.
			\end{split}
		\end{align}
		On the other hand, we have 
		\begin{align}
			\label{rank-equality-computation-end-of-proof-2to1}
			\begin{split}
				\Rank \mathcal{H}_{x-t}(n)
                    &= \Rank \begin{pmatrix}
					\mathcal{H}_{x-t}(n-1)  & 
                        {((x-t)\cdot\mathbf{v})}_{n;1}^{(n-1)} & {((x-t)\cdot\mathbf{v})}_{n;2}^{(n-1)}\\
					\big({((x-t)\cdot\mathbf{v})}_{n;1}^{(n-1)}\big)^T & \widehat{Z}_1 & \widehat{Z}_2 \\
					\big({((x-t)\cdot\mathbf{v})}_{n;2}^{(n-1)}\big)^T & \widehat{Z}_2^T & \widehat{Z}_3
				\end{pmatrix} \\
				&\underbrace{=}_{
					\substack{\eqref{definition-hat-K-2},\\\eqref{equality-v3}}}
				\Rank \begin{pmatrix}
					\mathcal{H}_{x-t}(n-1) & {((x-t)\cdot\mathbf{v})}_{n;1}^{(n-1)}\\
					\big({((x-t)\cdot\mathbf{v})}_{n;1}^{(n-1)}\big)^T & \widehat{Z}_1\\
					\big({((x-t)\cdot\mathbf{v})}_{n;2}^{(n-1)}\big)^T & \widehat{Z}_2^T
				\end{pmatrix} \\
				&\underbrace{=}_{
					\substack{\eqref{definition-hat-K-2},\\\eqref{equality-v3}}}
				\Rank \begin{pmatrix}
					\mathcal{H}_{x-t}(n-1) & {((x-t)\cdot\mathbf{v})}_{n;1}^{(n-1)}\\
					\big({((x-t)\cdot\mathbf{v})}_{n;1}^{(n-1)}\big)^T & \widehat{Z}_1
				\end{pmatrix} \\
				&\underbrace{=}_{\eqref{choice-of-hat-Z1}}
				np + p - m
			\end{split}
		\end{align}
		Combining \eqref{rank-inequality-computation-end-of-proof-2to1} and \eqref{rank-equality-computation-end-of-proof-2to1}, we get $m \geq m'$, which is a contradiction with \eqref{contradiction-argument}. Therefore
        $\mult_\mu t=m$ and $g(t) \neq 0$. This completes the proof.
	\end{proof}

  \begin{proof}[Proof of Corollary \ref{co:corollary-atom-avoidance}]
    Let 
        $\mathbf{v}_i^{(j)}$ be as in 
            \eqref{def:columns-mm} 
        and 
        ${((x-t)\cdot\mathbf{v})}_{i}^{(j)}$
        as in 
            \eqref{def:columns-loc-mm}.
		Since
		$
		M(n)
            $
		is invertible,
		it follows that the matrix
		\begin{equation*}
			\cM:=
                    \begin{pmatrix}
				\mathbf{v}_{0}^{(n)} &
				\begin{array}{c}
					\mathcal{H}_{x-t}(n-1) \\ 
                        {\big({((x-t)\cdot\mathbf{v})}_{n}^{(n-1)}\big)}^T
				\end{array}
			\end{pmatrix}
		\end{equation*}
		is also invertible.
            Therefore
				$$\Rank\begin{pmatrix}
					\mathcal{H}_{x-t}(n-1) \\ 
                        {\big({((x-t)\cdot\mathbf{v})}_{n}^{(n-1)}\big)}^T
				\end{pmatrix}
                =np,
                $$
            whence $\Rank\mathcal{H}_{x-t}(n-1) \geq (n-1)p.$
            By Theorem \ref{th:mainTheorem}, 
            the corollary follows.
    \end{proof}

  \begin{proof}[Proof of Corollary \ref{co:Simonov-pd}]
    Define a sequence 
        $\widetilde{\cS}=(\widetilde{S}_0,\widetilde{S}_1,\ldots,\widetilde{S}_{2n_1+2n_2})$,
    where
    $\widetilde S_i:=S_{i-2n_1}$. By assumption \eqref{strong-moment-matrix},
    $M_{\widetilde{S}}(n_1+n_2):=(\widetilde{S}_{i+j-2})_{i,j=1}^{n_1+n_2+1}$
    is positive definite.
    By Corollary \ref{co:corollary-atom-avoidance},
    $\widetilde \cS$ has a minimal representing measure $\widetilde \mu=\sum_{j=1}^{\ell}A_j\delta_{x_j}$
    for some $x_j\in \RR\setminus \{0\}$ and $A_j\in S_{p}^{\succeq 0}(\RR)$.
    Namely,
    $\widetilde S_i=\sum_{j=1}^{\ell} A_j x_j^{i}$
    for each $i=0,1,\ldots,2n_1+2n_2$.
    But then 
    $$
    S_i=\widetilde S_{i+2n_1}
        =\sum_{j=1}^{\ell} A_j x_j^{i+2n_1}
        =\sum_{j=1}^{\ell} (A_jx_j^{2n_1}) x_j^{i},
    $$
    whence $\mu:=\sum_{j=1}^{\ell}(A_jx_j^{2n_1})\delta_{x_j}$
    is a representing measure in Corollary \ref{co:Simonov-pd}. 
    \end{proof}

	\begin{remark}
		\label{re:remarkForTheoremTwoToOne}
		\begin{enumerate}[leftmargin=*]
			\item
			\label{re:remarkForTheoremTwoToOne-pt1}
			The polynomial $H(x)$ (see \eqref{gen-poly}), which is a block column relation of the matrix in \eqref{def:extended-matrix},
			can also be obtained by computing $$\begin{pmatrix}
				H_0^T & H_1^T & \cdots & H_n^T
			\end{pmatrix}^T := M(n)^{-1} \mathbf{v}_{n+1}^{(n)},$$ 
			to obtain $H(x) = x^{n+1}I_p - \sum_{i=0}^{n}x^iH_i.$
			\item
			The zeroes of the polynomial $g(x)$ from \eqref{eq:determinant}
			correspond to the other atoms in the representing measure, while the multiplicity of the atom as the zero of $g(x)$ coincides with the multiplicity of the atom.
			\item  
			\label{re:remarkForTheoremTwoToOne-pt2}
			Assume a linear operator $L:\RR[x]_{\leq 2n}\to \Sym_p(\RR)$
			has a representing measure
			$\mu=\sum_{j=1}^\ell \delta_{x_j}A_j$,
			where the atoms $x_j\in \RR$ are pairwise distinct, $A_j\in \Sym_p^{\succeq 0}(\RR)$
			and
			$\sum_{j=1}^\ell \Rank A_j = (n+1)p$.
			Assume that we know the atoms $x_1,x_2,\ldots,x_\ell$.
			It remains to compute the masses $A_j$.
			We denote by $V \equiv V_{(x_1,x_2,\ldots,x_\ell)} := \left (x_j^{i-1}\right )_{i,j = 1}^\ell$ the Vandermonde matrix. Since $x_1, x_2, \ldots, x_\ell$ are pairwise distinct, it follows that $V$ is invertible.
			The masses $A_j$ are obtained via 
			$$\begin{pmatrix}
				A_1 & A_2 & \cdots & A_\ell
			\end{pmatrix}^T = \left (V^{-1} \otimes I_p\right )
			\mathbf{v}_{0}^{(\ell-1)},$$ 
			where $\otimes$ denotes the Kronecker product of two matrices,
			i.e., $V^{-1}\otimes I_p=(V\otimes I_p)^{-1}=\big((x_{j}^{i-1}I_p)_{i,j=1}^{\ell}\big)^{-1}$
			. Note that if $\ell > 2n+2$, then not all $S_j$ are given. In particular, $S_{2n+2}, S_{2n+3}, \ldots, S_{\ell - 1}$ need to be computed recursively by 
			$$S_j = \begin{pmatrix}
				S_{j-n-1} & S_{j-n} & \cdots & S_{j-1}
			\end{pmatrix}  \begin{pmatrix}
				H_0^T & H_1^T & \cdots & H_n^T
			\end{pmatrix}^T,$$ 
			for $j = 2n+2, 2n+3, \ldots, \ell - 1$,
			where $H_i$ are as in \eqref{re:remarkForTheoremTwoToOne-pt1} above.
			\smallskip
			
			\item
			\label{re:remarkForTheoremTwoToOne-pt3}
			If $m=p$ in Theorem \ref{th:mainTheorem}, then $k$ must be 0 in \eqref{rank-increasement} and there are no blocks ${\widehat{((x-t)\cdot\mathbf{v})}}_{n;1}^{(n-1)}$ (see \eqref{def:widehatK}) and $\widehat{Z}_1$, $\widehat{Z}_2$ (see \eqref{def:Z}). Moreover, $k=0$ implies that
			$$\Rank \begin{pmatrix}
				\mathcal{H}_{x-t}(n-1) & {((x-t)\cdot\mathbf{v})}_{n}^{(n-1)}
			\end{pmatrix} = 
			\Rank \mathcal{H}_{x-t}(n-1)
			=np,$$
			whence $\mathcal{H}_{x-t}(n-1)$ is invertible.
			Further, $J$ in \eqref{definition-hat-K-2}
			is equal to 
                $$J = \mathcal{H}_{x-t}(n-1)^{-1}{((x-t)\cdot\mathbf{v})}_{n}^{(n-1)},$$ 
                while
			$Z$ in \eqref{definition-hat-K-2} is equal to 
                $$
                Z = {\big({((x-t)\cdot\mathbf{v})}_{n}^{(n-1)}\big)}^T(\mathcal{H}_{x-t}(n-1))^{-1}{((x-t)\cdot\mathbf{v})}_{n}^{(n-1)}.
                $$ 
                Therefore, the measure $\mu$ for $L$, with 
                $\mult_\mu t=m$, is unique. 
			\smallskip
			
			\item 
			If $m<p$ in Theorem \ref{th:mainTheorem},
			then $k>0$ in \eqref{rank-increasement} and we have a free choice of selecting $\widehat{Z}_1 \in \Sym_{p-m}(\RR)$ and different possibilities for $J$ in \eqref{definition-hat-K-2}. To be precise, $J$ can be chosen arbitrarily from the set 
			\begin{align*} 
                &\Big\{\begin{pmatrix}
				\mathcal{H}_{x-t}(n-1) & {\widehat{((x-t)\cdot\mathbf{v})}}_{n;1}^{(n-1)}
			\end{pmatrix}^\dag{\widehat{((x-t)\cdot\mathbf{v})}}_{n;2}^{(n-1)} + U \colon U
			\in M_{(n+1)p-m, m}(\RR)\\
            &\hspace{3cm}
			\text{ such that }
			\begin{pmatrix}
				\mathcal{H}_{x-t}(n-1) & {\widehat{((x-t)\cdot\mathbf{v})}}_{n;1}^{(n-1)}
			\end{pmatrix}U=\mathbf{0}_{np,m}\Big\},
                \end{align*}
			where $(\ast)^\dag$ denotes the Moore-Penrose pseudoinverse of the matrix $(\ast)$. Therefore, in this case, a measure  $\mu$ for $L$ such that
            $\mult_\mu t=m$ is not unique, as can be seen in 
			Example \ref{ex:non-uniqueness} below.
		\end{enumerate}
	\end{remark}

%% file: Examples.tex
    \section{Examples}
    \label{sec:examples}

    In this section we demonstrate the application of Theorem \ref{th:mainTheorem} on numerical examples.\\

    The following example considers a moment sequence 
    $\mathcal{S}$ with $k > 0$ as defined in \eqref{rank-increasement}. 
    We construct two distinct $(n+1)p$--atomic representing measures for $\mathcal{S}$. In both cases, the measures include $0$ in the support with largest multiplicity allowed by Theorem \ref{th:mainTheorem}, namely $m = p - k$, demonstrating that a representing measure for $\mathcal{S}$ containing an atom $t$ with $\mult_\mu t=m$ is not unique whenever $m < p$.
    
	\begin{example}{\footnote{The \textit{Mathematica} file with numerical computations can be found on the link \url{https://github.com/ZobovicIgor/Matricial-Gaussian-Quadrature-Rules/tree/main}.}}
		\label{ex:non-uniqueness}
		Let $p = 2$, $n = 1$ 
		and 
		$$S_0 = \begin{pmatrix}
			18 & 10 \\
			10 & 7
		\end{pmatrix}, 
		\quad
		S_1 = \begin{pmatrix}
			2 & 2 \\
			2 & 2
		\end{pmatrix}, 
		\quad 
		S_2 = \begin{pmatrix}
			50 & 26 \\
			26 & 14
		\end{pmatrix}.$$ 
		We can easily check that $M(1) \succ 0$. Let $t = 0$. We have that 
        $$\mathcal{H}_x(0) = \begin{pmatrix}
			S_1
		\end{pmatrix} = \begin{pmatrix}
			2 & 2 \\
			2 & 2
		\end{pmatrix} \quad \mathrm{and} \quad 
            {(x\cdot\mathbf{v})}_1^{(0)} = \begin{pmatrix}
			S_2
		\end{pmatrix} = \begin{pmatrix}
			50 & 26 \\
			26 & 14
		\end{pmatrix}.$$
		We observe that $\Rank  \mathcal{H}_x(0) = 1$ and $\Rank \begin{pmatrix}
			\mathcal{H}_x(0) & {(x\cdot\mathbf{v})}_1^{(0)}
		\end{pmatrix} = 2$, therefore $k = 1$ in \eqref{rank-increasement}. In this case, we can take a trivial  permutation $P=I_2$ in \eqref{def:widehatK} since 
            ${(x\cdot\mathbf{v})}_1^{(0)} = 
            \begin{pmatrix}
			  {(x\cdot\mathbf{v})}_{1;1}^{(0)} & 
                {(x\cdot\mathbf{v})}_{1;2}^{(0)}
		\end{pmatrix}$, where 
            ${(x\cdot\mathbf{v})}_{1;1}^{(0)} = \begin{pmatrix}
			50 \\
			26
		\end{pmatrix}$, satisfies
		$$
		\Rank \begin{pmatrix}
			\mathcal{H}_{x}(0) &  {(x\cdot\mathbf{v})}_1^{(0)}
		\end{pmatrix}=
		\Rank \begin{pmatrix}
			\mathcal{H}_x(0) &  {(x\cdot\mathbf{v})}_{1;1}^{(0)}
		\end{pmatrix}
		.$$ 
		Let $J := \begin{pmatrix}
			-\frac{1}{2} &
			1 &
			\frac{1}{2}
		\end{pmatrix}^T$. We check that 
        $$ {(x\cdot\mathbf{v})}_{1;2}^{(0)} = 
            \begin{pmatrix}
			26 \\
			14
		\end{pmatrix} = 
            \begin{pmatrix}
			\mathcal{H}_x(0) &  {(x\cdot\mathbf{v})}_{1;1}^{(0)}
		\end{pmatrix}J.
        $$
		We will now construct the matrix $Z = \widehat{Z}$
        (see \eqref{def:Z}), which is used in the proof of Theorem \ref{th:mainTheorem} to obtain a {polynomial $H(x)$ 
        (see \eqref{gen-poly}),  being a block column relation of 
			$\begin{pmatrix}
				M(1) & \mathbf{v}_{2}^{(1)}
			\end{pmatrix}
			$ and such that $\cZ(\det H(x))$ is precisely the set of atoms in some minimal representing measure for $\cS:=(S_0,S_1,S_2)$.} Note that since $t = 0$, we have $S_3 = Z$. For every $Z_1 \in \RR$, the matrix $$S_3 = Z = \begin{pmatrix}
			Z_1 & Z_2 \\
			Z_2^T & Z_3
		\end{pmatrix},$$ where $Z_2 = \begin{pmatrix}
			\big({(x\cdot\mathbf{v})}_{1;1}^{(0)}\big)^T & Z_1
		\end{pmatrix}J$ and $Z_3 = \begin{pmatrix}
			\big({(x\cdot\mathbf{v})}_{1;2}^{(0)}\big)^T & Z_2^T
		\end{pmatrix}J$,
		is symmetric and satisfies \eqref{eq:definition-of-widehatZ}.
		
		Let $Z_1^{(1)} := 2$ and $Z_1^{(2)} := 98$. Computing 
		$$
		Z_2^{(i)} = \begin{pmatrix}
			\big({(x\cdot\mathbf{v})}_{1;1}^{(0)}\big)^T & Z_1^{(i)}
		\end{pmatrix}J
		\quad \text{and} \quad
		Z_3^{(i)} = \begin{pmatrix}
			\big({(x\cdot\mathbf{v})}_{1;2}^{(0)}\big)^T & (Z_2^{(i)})^T
		\end{pmatrix}J
		$$ 
		for $i = 1,2$, we get
		$S_3^{(1)} = Z^{(1)} = \begin{pmatrix}
			2 & 2 \\
			2 & 2
		\end{pmatrix}$ and $S_3^{(2)} = Z^{(2)} = \begin{pmatrix}
			98 & 50 \\
			50 & 26
		\end{pmatrix}$. We can obtain the coefficients $H_0^{(i)}, H_1^{(i)}$ of the corresponding matrix polynomials 
		$$H^{(i)}(x) = x^2I_2 - xH_1^{(i)}-H_0^{(i)}$$ 
		by computing (see Remark \ref{re:remarkForTheoremTwoToOne}.\eqref{re:remarkForTheoremTwoToOne-pt1})
		$$\begin{pmatrix}
			H_0^{(i)} \\
			H_1^{(i)}
		\end{pmatrix} = M(1)^{-1} \cdot \begin{pmatrix}
			S_2 \\
			S_3^{(i)}
		\end{pmatrix}
		$$ 
		for $i = 1,2$. The polynomials are the following: 
		\begin{align*} 
			H^{(1)}(x) 
			&= x^2I_2 - x\begin{pmatrix}
				2 & \frac{1}{2} \\
				-4 & -1
			\end{pmatrix}
			-
			\begin{pmatrix}
				3 & \frac{3}{2} \\
				0 & 0
			\end{pmatrix} 
			, \\ 
			H^{(2)}(x) 
			&= x^2I_2 - x\begin{pmatrix}
				6 & \frac{5}{2} \\
				-8 & -3
			\end{pmatrix}-
			\begin{pmatrix}
				3 & \frac{3}{2} \\
				0 & 0
			\end{pmatrix},
		\end{align*}
		with the determinants
		\begin{align*}
			\det H^{(1)}(x) 
			&= x(x - 1)(x - \sqrt{3})(x + \sqrt{3}),\\
			\det H^{(2)}(x) 
			&= x(x - 3)(x - 1)(x + 1).
		\end{align*}
		Therefore the sets $\{0, 1, \sqrt{3}, -\sqrt{3}\}$ and $\{0, 3, 1, -1\}$ represent the atoms of two distinct matrix measures for $(S_0, S_1, S_2)$. Note that both determinants only have zeroes of multiplicity $1$, therefore the multiplicities of all the atoms from both sets are $1$. We confirm this by computing the corresponding masses for both sets of atoms. It turns out (using Remark
		\ref{re:remarkForTheoremTwoToOne}.\eqref{re:remarkForTheoremTwoToOne-pt2}) that the masses for the atoms $0, 1, \sqrt{3}, -\sqrt{3}$ in the first measure are 
		$\begin{pmatrix}
			0 & 0 \\
			0 & 1
		\end{pmatrix}$, 
		$\begin{pmatrix}
			2 & 2 \\
			2 & 2
		\end{pmatrix},$ 
		$\begin{pmatrix}
			8 & 4 \\
			4 & 2
		\end{pmatrix},$ 
		$\begin{pmatrix}
			8 & 4 \\
			4 & 2
		\end{pmatrix}$, respectively, and the masses for the atoms $0, 3, 1, -1$ in the second measure are $\begin{pmatrix}
			0 & 0 \\
			0 & 1
		\end{pmatrix},$ 
		$\begin{pmatrix}
			4 & 2 \\
			2 & 1
		\end{pmatrix},$ 
		$\begin{pmatrix}
			2 & 2 \\
			2 & 2
		\end{pmatrix},$ 
		$\begin{pmatrix}
			12 & 6 \\
			6 & 3
		\end{pmatrix}$, respectively.
	\end{example}
\bigskip

The next example illustrates that the inequality in Theorem \ref{th:mainTheorem}.\eqref{th:mainTheorem-pt2} can be strict. Namely, starting from a measure whose atom $0$ has multiplicity strictly smaller than $\Rank \mathcal{H}_x(n-1) - (n-1)p$, we build a new representing measure in which the multiplicity of the atom $0$ is the highest possible, 
i.e., equal to $\Rank \mathcal{H}_x(n-1) - (n-1)p$. 

\begin{example}{\footnote{The \textit{Mathematica} file with numerical computations can be found on the link \url{https://github.com/ZobovicIgor/Matricial-Gaussian-Quadrature-Rules/tree/main}.}}
\label{ex:2}
	Let $\mu=\sum_{j=1}^4 \delta_{x_j}A_j$ be a finitely atomic matrix measure with
	$(x_1, x_2, x_3, x_4) = (0, 1, -1, -2)$ and $(A_1, A_2, A_3, A_4) = \left (\begin{pmatrix}
		2 & 2 \\
		2 & 2
	\end{pmatrix}, \begin{pmatrix}
		1 & 1 \\
		1 & 1
	\end{pmatrix}, \begin{pmatrix}
		0 & 0 \\
		0 & 1
	\end{pmatrix}, \begin{pmatrix}
		1 & 0 \\
		0 & 0
	\end{pmatrix}\right )$, and let $L$ be a linear operator, defined by $L(p)=\int_{\RR} p\; d\mu$ for every $p\in \RR[x]_{\leq 2}$. We define 
	$$
	S_0 := 
	L(1)=
	\begin{pmatrix}
		4 & 3 \\
		3 & 4
	\end{pmatrix},\quad
	S_1 := 
	L(x)=
	\begin{pmatrix}
		-1 & 1 \\
		1 & 0
	\end{pmatrix}\quad\text{and}\quad
	S_2:=
	L(x^2)=
	\begin{pmatrix}
		5 & 1 \\
		1 & 2
	\end{pmatrix}.
	$$
	The measure $\mu$ contains the atom $0$ 
    with $\mult_\mu 0=\Rank A_1 = 1$.
	However, the localizing matrix 
    $\cH_x(0)= \begin{pmatrix}
		S_1
	\end{pmatrix}$ is invertible, therefore $$\Rank \begin{pmatrix}
		\mathcal{H}_x(0) & S_2
	\end{pmatrix} = \Rank \mathcal{H}_x(0) = 2 < \Rank \mathcal{H}_x(0) + 2 - \mult_\mu 0 = 3.$$ Since $\Rank \begin{pmatrix}
		\mathcal{H}_x(0) & S_2
	\end{pmatrix} = \Rank \mathcal{H}_x(0)$, it follows from Theorem \ref{th:mainTheorem} that there exists a $4$--atomic representing matrix measure $\widetilde{\mu}$ for $L$ which contains the atom $0$ with $\mult_{\widetilde \mu}0=2$. Such measure $\widetilde{\mu}$ is unique (see Remark \ref{re:remarkForTheoremTwoToOne}.\eqref{re:remarkForTheoremTwoToOne-pt3}) and we will now find its atoms. We first compute
	$$S_3^{(\widetilde{\mu})} := S_2^TS_1^{-1}S_2 = \begin{pmatrix}
		5 & 1 \\
		1 & 2
	\end{pmatrix}^T \cdot \begin{pmatrix}
		-1 & 1 \\
		1 & 0
	\end{pmatrix}^{-1} \cdot \begin{pmatrix}
		5 & 1 \\
		1 & 2
	\end{pmatrix} = \begin{pmatrix}
		11 & 13 \\
		13 & 8
	\end{pmatrix}.$$
	Then we obtain the {polynomial $H(x)$,  which is a block column relation of 
		$\begin{pmatrix}
			M(1) & \mathbf{v}_2^{(1)}
		\end{pmatrix}
		$ and such that $\cZ(\det H(x))$ is precisely the set of atoms in some minimal represenitng measure for $L$.} Namely, $H(x) = x^2I_2 - xH_1-H_0$
	where
	$$\begin{pmatrix}
		H_0 \\
		H_1
	\end{pmatrix} = M(1)^{-1} \begin{pmatrix}
		S_2 \\
		S_3^{(\widetilde{\mu})}
	\end{pmatrix} = \begin{pmatrix}
		\begin{pmatrix}
			0 & 0 \\
			0 & 0
		\end{pmatrix} \\
		\begin{pmatrix}
			1 & 2 \\
			6 & 3
		\end{pmatrix}
	\end{pmatrix}.$$ Thus, it follows that $$H(x) = x^2I_2 - x\begin{pmatrix}
		1 & 2 \\
		6 & 3
	\end{pmatrix}.$$ The atoms of the measure $\widetilde{\mu}$ are the zeroes of
	$$\det H(x) = x^2(x^2 - 4x - 9) = x^2(x - 2 + \sqrt{13})(x - 2 - \sqrt{13}),$$ therefore $$\widetilde{\mu} = \delta_{0}B_1 + \delta_{2 - \sqrt{13}}B_2 + \delta_{2 + \sqrt{13}}B_3,$$ where $\Rank B_1 = 2$ and $\Rank B_2 = \Rank B_3 = 1$. By Remark \ref{re:remarkForTheoremTwoToOne}.\eqref{re:remarkForTheoremTwoToOne-pt2}, the masses of the atoms are 
	$$B_1 = \begin{pmatrix}
		3 & \frac{10}{3} \\[0.3em]
		\frac{10}{3} & \frac{34}{9}
	\end{pmatrix},\quad B_2 = \begin{pmatrix}
		\frac{13+3\sqrt{13}}{26} & \frac{-13-5\sqrt{13}}{78} \\[0.3em]
		\frac{-13-5\sqrt{13}}{78} & \frac{13+2\sqrt{13}}{117}
	\end{pmatrix},\quad B_3 = \begin{pmatrix}
		\frac{13-3\sqrt{13}}{26} & \frac{5\sqrt{13}-13}{78} \\[0.3em]
		\frac{5\sqrt{13}-13}{78} & \frac{13-2\sqrt{13}}{117}
	\end{pmatrix}.$$
\end{example}

%% file: Main-result-generalized-version.tex
\section{Generalized matricial Gaussian quadrature rules with prescribed atom}
\label{sec:posDef-generalized}

In this section 
 we allow the evaluation at $\infty$ (see \eqref{subsec:evaluation-at-infty}) as a measure and prove a sufficient condition for the existence of a generalized matricial Gaussian quadrature rule for a linear operator $L:\RR[x]_{\leq 2n}\to \Sym_p(\RR)$, containing $\Rank M(n-1)$ real atoms, among which a prescribed atom has a prescribed multiplicity (see Theorem \ref{th:WithAtomInfinity}).\\

Let $m,n\in \NN$ and 
	\begin{equation*}
		\cM=\left( \begin{array}{cc} A & B \\ C & D \end{array}\right)\in M_{n+m}(\RR),
	\end{equation*}
where $A\in \RR^{n\times n}$, $B\in \RR^{n\times m}$, $C\in \RR^{m\times n}$  and $D\in \RR^{m\times m}$.
The \textbf{Schur complement} \cite{Zha05} of $D$ in $\cM$ is defined by 
    $\cM\big/D=A-BD^{-1} C$. 

\begin{theorem}
\label{th:WithAtomInfinity}
  Let $n, p \in \NN$ and $L:\RR[x]_{\leq 2n}\to \SSS_p(\RR)$
  be a linear operator such that $M(n-1)$ is positive definite.
        Fix $t\in \RR$ and $m\in \NN\cup\{0\}$.
        Assume the notation from \S\ref{sec:prel}.
        If 
        \begin{equation} 
            \label{cond:generalized-cond-1}
            m= \Rank 
                \begin{pmatrix}
                    \cH_{x-t}(n-2) & {(\mathbf{(x-t)\cdot v})}_{n-1}^{(n-2)}
                \end{pmatrix}
                -\Rank \cH_{x-t}(n-2)
        \end{equation}
        and
        \begin{equation} 
            \label{cond:generalized-cond-2}
            M(n)\big/ M(n-1)\succeq 0,
        \end{equation}
        then there exists a $(\Rank M(n))$--atomic 
            $(\RR\cup \{\infty\})$--representing measure $\mu$
            for $L$ such that $\mult_\mu t=m$
            and $\mult_{\mu}\infty=\Rank M(n)-np$.
\end{theorem}

\begin{proof}
    By the same proof as for the implication 
    $
    \eqref{th:mainTheorem-pt2}
    \Rightarrow
    \eqref{th:mainTheorem-pt1}
    $
    of
    Theorem \ref{th:mainTheorem},  \eqref{cond:generalized-cond-1} implies that the sequence
    $\cS^{(2n-1)}:=(S_0,S_1,\ldots,S_{2n-1})$
    has a $(\Rank M(n-1))$--atomic $\RR$--representing measure $\widetilde\mu$ such that $\mult_{\widetilde{\mu}}t=m$.
    Let $\widetilde S_{2n}=\int_\RR x^{2n}d\widetilde\mu$
    and
    $\widetilde{\cS}=(S_0,S_1,\ldots,S_{2n-1},\widetilde{S}_{2n})$.
    Since 
    $\Rank M_{\widetilde{\cS}}(n)=
    \Rank M_{\widetilde{\cS}}(n-1)$, it follows that
    $M_{\widetilde{\cS}}(n)\big/M_{\widetilde{\cS}}(n-1)=\mathbf{0}_{p}$. Moreover,
    $$
    M(n)\big/M(n-1)=
    S_{2n}-\widetilde S_{2n}+
    M_{\widetilde{\cS}}(n)\big/M_{\widetilde{\cS}}(n-1).
    $$
    By \eqref{cond:generalized-cond-2},
    $S_{2n}-\widetilde S_{2n}\succeq 0$ and
    thus $\cS=(S_0,S_1,\ldots,S_{2n})$ has a $(\RR\cup \{\infty\})$--representing measure
    $$\mu:=\widetilde{\mu}+\delta_{\infty} \big(M(n)\big/M(n-1)\big).$$
    This concludes the proof of Theorem \ref{th:WithAtomInfinity}.
\end{proof}